\documentclass[11pt,reqno]{amsart}

\usepackage{amsmath, amsthm, amssymb}
\usepackage{amsfonts}
\usepackage[ansinew]{inputenc}
\usepackage[dvips]{epsfig}
\usepackage{graphicx}
\usepackage[english]{babel}
\usepackage{comment}
\usepackage{graphicx}
\usepackage{hyperref}
\usepackage{bbm} % for the characteristic function \mathbbm{1}
\usepackage{cite}
\usepackage{graphicx}
\usepackage{amscd}
\usepackage{xcolor}
\usepackage{bm}
\usepackage{enumerate}

\usepackage{verbatim}
\usepackage{hyperref}
\usepackage{amstext}
\usepackage{latexsym}
\let\oldsqrt\sqrt
\def\sqrt{\mathpalette\DHLhksqrt}
\def\DHLhksqrt#1#2{%
\setbox0=\hbox{$#1\oldsqrt{#2\,}$}\dimen0=\ht0
\advance\dimen0-0.2\ht0
\setbox2=\hbox{\vrule height\ht0 depth -\dimen0}%
{\box0\lower0.4pt\box2}}

\allowdisplaybreaks

\newcommand{\R}{\mathbb{R}} % reelle Zahlen
\newcommand{\N}{\mathbb{N}} % natuerliche Zahlen
\newcommand{\dist}{\textnormal{dist}} % dist ...
\newcommand{\diam}{\textnormal{diam}} % diam ...
\newcommand{\supp}{\textnormal{supp}} % supp ...
\newcommand{\ov}{\overline}

\renewcommand{\phi}{\varphi}

\newcommand{\cK}{{\mathcal K}}

\newcommand{\cX}{{\mathcal X}}

\theoremstyle{definition}
\newtheorem{defi}{Definition}[section]
\newtheorem{remark}[defi]{Remark}

\theoremstyle{plain} %default%plain
\newtheorem{thm}[defi]{Theorem}
\newtheorem{prop}[defi]{Proposition}
\newtheorem{lemma}[defi]{Lemma}

\theoremstyle{definition}

\numberwithin{equation}{section}

 \title[Mixed local and nonlocal supercritical Dirichlet problems]{Mixed local and nonlocal supercritical Dirichlet problems} 

\author[David Amundsen, Abbas Moameni and Remi Yvant Temgoua]{David Amundsen$^1$, Abbas Moameni$^1$ and Remi Yvant Temgoua$^1$}
\address{$^1$ School of Mathematics and Statistics, Carleton University, Ottawa, Ontario, Canada.}

\email{dave@math.carleton.ca} 

\email{momeni@math.carleton.ca}

\email{remiyvanttemgoua@cunet.carleton.ca}

\date{\today}

\begin{document}

\begin{abstract}
	In this work, we consider a mixed local and nonlocal Dirichlet problem with supercritical nonlinearity. We first  establish a multiplicity result for the problem
	\begin{equation}\label{e1}
	Lu=|u|^{p-2}u+\mu |u|^{q-2}u~~~\text{in}~~\Omega,\quad\quad u=0~~~\text{in}~~\R^N\setminus\Omega,
	\end{equation}
where $L:=-\Delta +(-\Delta)^s$ for  $s\in(0,1)$ and   $\Omega\subset\R^N$ is a bounded domain.  Precisely, we show that
%there exists $\mu_\#>0$ such that 
problem \eqref{e1}  for $1<q<2<p$ has a positive solution as well as  a sequence of sign-changing solutions with a negative energy for small values of $\mu$.  Here $u$ can be either a scalar function,  or a vector valued function  so that (\ref{e1}) turns into a system with supercritical nonlinearity.
%for $\mu\in(0,\mu_\#)$. 
Moreover, whenever the domain is symmetric, we also prove the existence of symmetric solutions enjoying the same symmetry properties.   We shall also prove an existence result for the supercritical Hamiltonian system 
\[Lu=|v|^{p-2}v, \qquad 
        Lv=|u|^{d-2}u+\mu |u|^{q-2}u\]
with the Dirichlet boundary condition on $\Omega$ where  $1<q<2<p, d$.
Our method is variational, and in  both problems the lack of compactness for the supercritical problem  is recovered by working on a closed convex subset of an appropriate  function space.

\end{abstract}

\maketitle

{\footnotesize
	\begin{center}
	
		\textit{Keywords.}  Variational principle, Symmetric solutions, Euler-Lagrange functional, Convex Analysis.\\[0.2cm]
		
		%\textit{MSC2020:} 35R11, 45C05.
		
	\end{center}
}

\section{Introduction and main results}\label{section:introduction}

Let $s\in(0,1)$, $1<q<2<p,d$ and $\Omega\subset\R^N, N>2$ be a bounded domain with $C^1$ boundary. The purpose of the present paper is to obtain the existence and qualitative properties of solutions to 
the equation
\begin{equation}\label{e2}
\left\{\begin{aligned}
Lu&=|u|^{p-2}u+\mu |u|^{q-2}u~~~\text{in}~~\Omega\\
u&=0\quad\quad\quad\quad\quad\quad\quad\quad~~\text{in}~~\R^N\setminus\Omega,
\end{aligned}
\right.
\end{equation}
and the Hamiltonian system
\begin{equation}\label{hamiltonian system0}
    \left\{ 
    \begin{aligned}
        Lu&=|v|^{p-2}v~~~~~~~~~~~~~~~~~~~\text{in}~~~\Omega\\
        Lv&=|u|^{d-2}u+\mu |u|^{q-2}u~~~~\text{in}~~~\Omega\\
        u&=v=0~~~~~~~~~~~~~~~~~~~~~~\text{in}~~~\R^N\setminus\Omega
    \end{aligned}
    \right.
\end{equation}
where $L:=-\Delta +(-\Delta)^s$ is the so-called mixed local and nonlocal operator and $\mu>0$ is a positive parameter. Recall that $-\Delta$ denotes the classical Laplace operator and $(-\Delta)^s$ the standard fractional Laplacian defined for every sufficiently regular function $u: \R^N\rightarrow\R$ by 
\begin{equation*}
(-\Delta)^su(x)=c_{N,s}P.V.\int_{\R^N}\frac{u(x)-u(y)}{|x-y|^{N+2s}}\ dy,~~~x\in\R^N,
\end{equation*}
where $c_{N,s}$ is a normalization constant and ``$P.V$'' stands for the Cauchy principal value.

Operators of the form $L=-\Delta+(-\Delta)^s$ naturally arise in the study of superposition of two stochastic processes: a classical stochastic process (Brownian motion) whose infinitesimal generator is $-\Delta$, and a stochastic process with long-jumps ($2s$-stable L\'{e}vy process) whose infinitesimal generator is $(-\Delta)^s$. We refer to \cite{dipierro2021description} for a complete exposition on the superposition of Brownian and L\'{e}vy processes. 

Very recently, the study of mixed local and nonlocal operator of the form $-\Delta+(-\Delta)^s$ has attracted much attention both from probabilistic and analytic points of view. Among others, this operator is a good candidate to describe many phenomena in nature since it takes into account local and nonlocal behavior of the system. A non-exhaustive list of references in which mixed local and nonlocal operators $L=-\Delta+(-\Delta)^s$ have been considered is \cite{abatangelo2021elliptic,biagi2022brezis,anthal2022choquard,biagi2022mixed,biagi2022hong,biagi2021faber,biagi2021semilinear,biagi2021brezis,biswas2021mixed,biswas2022boundary,byun2023regularity,byun2023mixed,de2022gradient,barles2012lipschitz,dipierro2022linear,Li, maione2022variational,salort2022mixed,fang2022regularity,garain2022regularity,garain2023higher,garain2022mixed}.

The convex-concave problem \eqref{e2} in the case when $L=-\Delta$ or $L=(-\Delta)^s$ has been widely studied in the literature. 
	Note that this  problem has  concave-convex non-linearties. In this direction with $L=-\Delta$, in a very first work done by Ambrosetti, Brezis and Cerami in \cite{ABC}, authors have proved the existence of a  positive solutions and infinitely many sign changing  with the non-linearity $u^p+\lambda u^q$ satisfying $0 < q < 1 < p.$  The result in \cite{ABC} has received a lot of attention because there was no control on $p$ from the above. This type of problem presents many difficulties since supercritical nonlinearity is involved: a standard argument in the calculus of variations cannot be applied to derive solutions since compact embedding fails to hold in this case. However, a  variational method was recently  developed  (see \cite{moameni2017variational,moameni2018variational}) that addresses the difficulty. It consists of restricting the Euler-Lagrange functional of the problem to an appropriate convex set. This argument has been successfully applied to derive a multiplicity result for problems of the type \eqref{e2} for both $L=-\Delta$ and $L=(-\Delta)^s$. We refer to \cite{kouhestani2019multiplicity,kouhestani2018multiplicity} and the references therein. 
%An interesting question therefore arises: can we obtain similar result to \cite{kouhestani2019multiplicity,kouhestani2018multiplicity} with the mixed operator $L=-\Delta+(-\Delta)^s$?

The aim of this paper is to investigate mixed local and nonlocal operator of the form $L=-\Delta+(-\Delta)^s$ where we have a convex-concave nonlinearity. Our result is new for supercritical nonlinearity. Let us point out that nonlinear problem of the form \eqref{e2} with subcritical nonlinearity has been recently considered in \cite{maione2022variational}. Even though we are stating our results for the case where $u:\Omega \to \R$ is a scalar function,  with some minor changes one can consider a system of equations where $u:\Omega \to \R^m$ is a vector function.   \\

Our first main result reads as follows.

\begin{thm}\label{first-main-result}
	Let $\Omega\subset\R^N$ be a bounded domain with $C^1$ boundary. If $N>2$ and $1<q<2<p$, then there exists $\mu_\#>0$ such that for every $\mu\in(0,\mu_\#)$ problem \eqref{e2} admits at least one positive solution in $\cX^1_0\cap L^{\infty}(\R^N)$ with negative energy.
\end{thm}
To be precise,  the solution $u$ obtained  in the above theorem and throughout the paper   for problem \eqref{e2} holds 
		in the weak sense i.e.,
		\begin{equation*}
		\int_{\Omega}\nabla u\cdot\nabla\phi\ dx+\int_{\R^N}\int_{\R^N}\frac{(u(x)-u(y))(\phi(x)-\phi(y))}{|x-y|^{N+2s}}\ dxdy=\int_{\Omega}(|u|^{p-2}u+\mu |u|^{q-2}u)\phi\ dx,
		\end{equation*}
for all $\phi\in \cX^1_0\cap L^p(\Omega).$	
As a byproduct of Theorem \ref{first-main-result}, one obtains a nonexistence result for large $\mu$, as given by the following result. Here, since problem \eqref{e2} is $\mu$-dependent, we use $\eqref{e2}_{\mu}$ in place of \eqref{e2} for clarity. 

\begin{thm}\label{non-existence-result}
	Let $\Omega\subset\R^N$ be a bounded domain with $C^1$ boundary. If $N>2$ and $1<q<2<p$, then there exists $\varLambda>0$ such that 
	\begin{itemize}
		\item [$(i)$] for all $\lambda\in(0,\varLambda)$ problem $\eqref{e2}_{\lambda}$ admits a positive solution in $\cX^1_0\cap L^{\infty}(\R^N)$.
		\item [$(ii)$] for all $\lambda>\varLambda$ problem $\eqref{e2}_{\lambda}$ admits no positive solution.
	\end{itemize}
\end{thm}

Our next result is concerned with the multiplicity of solutions. It reads as follows.
%It also shows that there is no positive solution for large value of $\mu$. 

\begin{thm}\label{second-main-result}
	Let $\Omega\subset\R^N$ be a bounded domain with $C^1$ boundary. If $N>2$ and $1<q<2<p$, then there exists $\mu_\#>0$ such that for every $\mu\in(0,\mu_\#)$ problem \eqref{e2} admits infinitely many distinct solutions in $\cX^1_0\cap L^{\infty}(\R^N)$ with negative energy. 
\end{thm}

Our last result shows that in the case when $\Omega$ is a symmetric domain,  solutions enjoy some symmetry properties. Before stating our last result we let
\begin{equation}\label{reflection}
\sigma_j: \R^N\rightarrow\R^N,~~~x=(x_1,\dots,x_j,\dots,x_N)\mapsto\sigma_j(x)=(x_1,\dots,-x_j,\dots,x_N)
\end{equation}
being the reflection with respect to the $j$-th coordinate.
Our symmetry theorem reads as follows.
\begin{thm}\label{third-main-result}
	Assume that $\Omega$ is a symmetric domain and let $\sigma_j$ be the reflection defined above. Then if $N>2$ and $1<q<2<p$, there exists $\mu_\#>0$ such that for every $\mu\in(0,\mu_\#)$ problem \eqref{e2} admits a nontrivial solution $u\in \cX^1_0\cap L^{\infty}(\R^N)$ satisfying
	\begin{equation}
	u(\sigma_j(x))=-u(x) \quad\quad\forall x\in\Omega, \quad\forall j=1,\dots,N.
	\end{equation}
\end{thm}
We briefly comment on the proof of Theorem \ref{first-main-result}. The general strategy, inspired by \cite{moameni2017variational,moameni2018variational} is to consider a restriction of the Euler-Lagrange functional corresponding to \eqref{e2} on a convex set. Next, we show that the restricted functional possesses a critical  point by using the mountain pass theorem. Finally, in the last step, we show that these critical points are in fact solutions to the original problem. The proof of Theorems \ref{non-existence-result}, \ref{second-main-result}, and \ref{third-main-result} follows basically the same lines of thought. \\
The main result for the Hamiltonian system \eqref{hamiltonian system0}  reads as follows.
\begin{thm}\label{main-result-hamiltonian-system}
	Let $\Omega\subset\R^N$ be a bounded domain with $C^{1,1}$ boundary, and let $p'=p/(p-1)$. If $N>2p'$ and $1<q<p'<2<p$, then there exists $\mu_*>0$ such that for every $\mu\in(0,\mu_*)$ problem \eqref{hamiltonian system0} has a positive weak solution $(u,v)$ in $\cX^1_0\cap W^{2,p'}(\Omega)\cap L^{\infty}(\R^N)$.
\end{thm}

We shall prove the above theorem by finding the critical points of the functional 
\begin{align*}
    I(u)=\frac{1}{p'}\int_{\Omega}|Lu|^{p'}\ dx-\frac{1}{d}\int_{\Omega}|u|^{d}\ dx-\frac{\mu}{q}\int_{\Omega}|u|^{q}\ dx
    \end{align*}
on an appropriate convex and closed subset of $L^{\infty}(\R^N).$\\

The paper is organized as follows. In Section \ref{section:preliminary-functional-setting} we introduce some general facts from convex analysis as well as the corresponding function setting of our problem. Section \ref{section:proof-of-main-result} is devoted to the proof of main results. In Section \ref{section:comments} we briefly comment on the non-homogeneous supercritical problem. Finally, in Section \ref{section:hamiltonian-system}, we consider the case of the  Hamiltonian system.\\

\section{Preliminary and Functional settings}\label{section:preliminary-functional-setting}

In this section, we introduce some well-known results and definitions from convex analysis. Moreover, we also introduce the usual  functional settings related to our purpose.

\subsection{Preliminary: Some properties from convex analysis}

Given a reflexive Banach space $V$ and its topology dual $V^*$, we let $\langle\cdot,\cdot\rangle$ be the duality pairing between $V$ and $V^*$. Denote by $\sigma(V,V^*)$ the weak topology induced by $V$. Let $\Psi: V\to\R$ be a function. We say that $\Psi$ is lower semi-continuous if for every sequence $u_i\in V$ with $u_i\rightarrow u$ in the weak topology $\sigma(V,V^*)$, one has

\begin{equation*}
\Psi(u)\leq\liminf_{i\to\infty}\Psi(u_i).
\end{equation*} 

Let now $\Psi:V\to\R\cup\{\infty\}$ be a proper (i.e. $Dom(\Psi)=\{u\in V: \Psi(u)<\infty\}\neq\emptyset$) convex function. The subdifferential $\partial\Psi(u)$ is the set-valued operator defined as follows. If $u\in Dom(\Psi)$,

\begin{equation*}
\partial\Psi(u)=\{u^*\in V^*:\langle u^*, v-u\rangle+\Psi(u)\leq \Psi(v)~~\text{for all}~~v\in V\},
\end{equation*}
and if $u\notin Dom(\Psi)$, $\partial\Psi(u)=\emptyset$. If $\Psi$ is G$\hat{\text{a}}$teaux differentiable at $u$ with $D\Psi(u)$ being the derivative of $\Psi$ at $u$, then from \cite{ekeland1976convex}, we have $\partial\Psi(u)=D\Psi(u)$. We then have the following

\begin{prop}\label{lower-semi-continuous-properti}
	Let $F:V\to\R\cup\{\infty\}$ be a convex function. If $F$ is G$\hat{\text{a}}$teaux differentiable function at $u\in V$, and if $u_i$ is a sequence weakly converging to $u$ in $V$ (i.e., $u_i\rightharpoonup u$ weakly in $V$), then
	
	\begin{equation*}
	F(u)\leq\liminf_{i\to\infty}F(u_i).
	\end{equation*}
\end{prop}

\begin{defi}
	The function $I=\Psi-\Phi$ on $V$ satisfies the hypothesis \textbf{(H)} if
	\begin{itemize}
		\item [$(i)$] $\Phi\in C^1(V,\R)$ and
		\item [$(ii)$] $\Psi:V\to(-\infty,+\infty]$ is proper, convex and lower semi-continuous.
	\end{itemize} 
\end{defi}

We now recall the following definition of a critical point of $I=\Psi-\Phi$ due to Szulkin, see \cite{szulkin1986minimax}.

\begin{defi}\label{def3}
	A point $u\in V$ is a critical point of $I=\Psi-\Phi$ if $u\in Dom(\Psi)$ and it satisfies the inequality
	
	\begin{equation}\label{critical-point}
	\langle D\Phi(u), u-v\rangle+\Psi(v)-\Psi(u)\geq0,~~\text{for all}~~v\in V.
	\end{equation}
\end{defi}

We also have the following

\begin{defi}
	Let $I=\Psi-\Phi$ and $c\in\R$. We say that $I$ satisfies the compactness condition of Palais-Smale type \textbf{(PS)} if every sequence $\{u_i\}$ such that $I(u_i)\to c$ and 
	
	\begin{equation}
	\langle D\Phi(u_i), u_i-v\rangle+\Psi(v)-\Psi(u_i)\geq-\varepsilon_i\|u_i-v\|_{V},~~~\text{for all}~~v\in V,
	\end{equation}
	
	where $\varepsilon_i\to0$, possesses a convergence subsequence.
\end{defi}

Let the genus of a set now be defined in the following way. This notion is crucial in obtaining multiplicity result. More details about the genus of a set can be found in \cite{ambrosetti1983some,rabinowitz1974variational}.

\begin{defi}
	Let $X=\{A\subset V: A~\text{is closed and}~A=-A\}$. For a nonempty set $A\in X$, we define genus $k$ (denoted $\gamma(A)=k$) to be the least integer $k$ such that there exists an odd and continuous map $h: A\to\R^k\setminus\{0\}$. If there is no such $k$, set $\gamma(A)=\infty$. In addition, $\gamma(\emptyset)=0$.
\end{defi}

\begin{prop}\label{genus}
	Assume that $W$ is a linear subspace of $V$ with $\dim W=k$. If $A\subset W$ is a symmetric bounded neighbourhood of $0$ in $W$, then $\gamma(A)=k$. 
\end{prop}

Define $Y=\{A\subset V: A~\text{is closed, bounded and}~ A\neq\emptyset\}$. Then by \cite[$\S15$, $VII$ and $\S29$, $IV$ ]{kuratowski1958topologie}, $Y$ is a complete metric space with the metric

\begin{equation*}
\dist(A,B)=\max\{\sup_{x\in A}\dist(x,B),\sup_{y\in B}\dist(y,A)\}.
\end{equation*}

Recall that $\dist(z,A)=\inf_{a\in A}|z-a|$. Let $Z=\{A\in Y: A~\text{is compact and symmetric}\}$ and define

\begin{equation}\label{a1}
Z_j=\overline{\{A\in Z: 0\notin A, ~\gamma(A)\geq j\}}^{Z}.
\end{equation}

Here, $\overline{U}^Z$ is the closure of $U$ in $Z$. Then, $(Z,\dist)$ and $(Z_j,\dist)$ are complete metric spaces.\\

We conclude this subsection with the following crucial result on critical points of even functions of type \textbf{(H)} stated in \cite{szulkin1986minimax}.

\begin{thm}\label{multiplicity-result}
	Suppose that $I:V\to (-\infty,+\infty]$ satisfies \textbf{(H)} and \textbf{(PS)}, $I(0)=0$ and $\Psi,~\Phi$ are even. Define
	
	\begin{equation*}
	c_j=\inf_{A\in Z_j}\sup_{u\in A}I(u).
	\end{equation*}
	
	If $-\infty<c_j<0$ for $j=1,\dots,k$, then $I=\Psi-\Phi$ has at least $k$ distinct pairs of nontrivial critical points in the sense of Definition \ref{def3}.
\end{thm}

\subsection{Functional settings} Here, we introduce the functional setting related to the operator $L=-\Delta+(-\Delta)^s$. For all $s\in(0,1)$ we set

\begin{equation*}
[u]^2_s:=\int_{\R^N}\int_{\R^N}\frac{(u(x)-u(y))^2}{|x-y|^{N+2s}}\ dxdy.
\end{equation*}

Throughout the paper, $\Omega\subset\R^N$ is a bounded domain with $C^1$ boundary. Let $\cX^1$ be the space defined as 

\begin{equation}
\cX^1=\{u:\R^N\to\R~\text{Lebesgue measurable}: u|_{\Omega}\in H^1(\Omega)~\text{and}~[u]_s<\infty\}.
\end{equation}

Then $\cX^1$ is a Banach space with the norm

\begin{equation*}
\|u\|^2_{\cX^1}=\|u\|^2_{H^1(\Omega)}+[u]^2_s.
\end{equation*}
Now, we set
\begin{equation}
\cX^1_0=\{u\in\cX^1: u=0~~\text{in}~~\R^N\setminus\Omega\}.
\end{equation}
Thanks to the regularity assumption on $\partial\Omega$, we have that every function $u\in\cX^1_0$ satisfies $u|_{\Omega}\in H^1_0(\Omega)$. Therefore, $\cX^1_0$ is a Hilbert space endowed with the norm

\begin{equation}\label{p_0}
\|u\|^2_{\cX^1_0}=\|\nabla u\|^2_{L^2(\Omega)}+[u]^2_s
\end{equation}
with scalar product

\begin{equation*}
\langle u,v\rangle_{\cX^1_0}=\int_{\Omega}\nabla u\cdot\nabla v\ dx+\int_{\R^N}\int_{\R^N}\frac{(u(x)-u(y))(v(x)-v(y))}{|x-y|^{N+2s}}\ dxdy\quad\text{for all}~~u,v\in\cX^1_0.
\end{equation*}
Notice that from \cite[Proposition 2.2]{di2012hitchhikers} there holds that
\begin{equation*}
    [u]^2_s\leq C\|\nabla u\|^2_{L^2(\Omega)}~~~\forall u\in\cX_0^1
\end{equation*}
for some positive constant $C>0$ independent on $u$. Therefore, thanks to \eqref{p_0}, we have\footnote{The notation $f\asymp g$ means that the two-sided estimate $C_1g\leq f\leq C_2g$ is true, where $C_1, C_2>0$ are two positive constants.}
\begin{equation}\label{p_1}
    \|u\|_{\cX_0^1}\asymp \|\nabla u\|_{L^2(\Omega)}~~~\forall u\in\cX_0^1.
\end{equation}
Let us now recall the maximum principle (see \cite[Theorem 1.2]{biagi2022mixed}) associated to $L$.

\begin{prop}\label{maximum-principle}
	Let $u\in H^1(\R^N)$ weakly satisfy $Lu\geq0$ in $\Omega$. If $u\geq0$ a.e. in $\R^N\setminus\Omega$, then $u\geq0$ a.e. in $\Omega$.
\end{prop}

We end this subsection with the following a priori $L^{\infty}$-estimate from \cite[Theorem 4.7]{biagi2022mixed}.

\begin{lemma}\label{l-infty-estimate}
	Assume that $N>2$ and let $f\in L^p(\Omega)$ with $p>\frac{N}{2}$. Let $u\in \cX^1_0$ be a weak solution of 
	\begin{equation}
	Lu=f~~\text{in}~~\Omega,\quad\quad u=0~~\text{in}~~\R^N\setminus\Omega.
	\end{equation}
	Then $u\in L^{\infty}(\R^N)$ and there exists a positive constant $C>0$ such that
	\begin{equation}
	\|u\|_{L^{\infty}(\R^N)}\leq C\|f\|_{L^p(\Omega)}.
	\end{equation}
\end{lemma}

\section{Proof of main results}\label{section:proof-of-main-result}

In this section, we present the variational scheme that allows us to obtain the existence of solutions for problem \eqref{e2}. %This new variational method was introduced by the second author in \cite{moameni2018variational,moameni2017variational}. it reads as follows. 

Let $V$ be reflexive Banach space and $K$ be a nonempty, convex, and weakly closed subset of $V$. Let $\Psi: V\to \R\cup\{+\infty\}$ be a proper, convex, and lower semi-continuous function which is differentiable a la G$\hat{\text{a}}$teaux on $K$. Define the restriction $\Psi_K$ of $\Psi$ on $K$ as

\begin{equation}
\Psi_K(u)=\left\{\begin{aligned}
&\Psi(u)~~\text{if}~~ u\in K;\\
&+\infty~~\text{if}~~ u\notin K.
\end{aligned}
\right.
\end{equation}
Given a function $\Phi\in C^1(V,\R)$, we let $I_K: V\to (-\infty,+\infty]$ be the functional defined by
\begin{equation*}
I_K(u):=\Psi_K(u)-\Phi(u).
\end{equation*}

We now recall the following pointwise condition introduced in   \cite{moameni2018variational}.

\begin{defi}\label{def}
	We say that the triple $(\Psi, K, \Phi)$ satisfies the pointwise invariance condition at a point $u_0\in V$ if there is a convex G$\hat{\text{a}}$teaux differentiable function $G: V\to\R$ and a point $v_0\in K$ such that
	\begin{equation*}
	D\Psi(v_0)+DG(v_0)=D\Phi(u_0)+DG(u_0).
	\end{equation*}
\end{defi}

We then have the following variational principle established  in \cite{moameni2018variational}.

\begin{thm}\label{key-theorem-in-the-variational-methods}
	Let $V$ be a reflexive Banach space and $K$ be a convex and weakly closed subset of $V$. Let $\Psi: V\to\R\cup\{+\infty\}$ be a convex, lower semi-continuous function which is G$\hat{\text{a}}$teaux differentiable on $K$ and let $\Phi\in C^1(V,\R)$. Assume that the following assertions hold:
	
	\begin{itemize}
		\item [$(i)$] The functional $I_K: V\to\R\cup\{+\infty\}$ defined by $I_K(u)=\Psi_K(u)-\Phi(u)$ has a critical point $u_0\in V$ in the sense of Definition \ref{def3}, and;
		\item [$(ii)$] the triple $(\Psi, K, \Phi)$ satisfies the pointwise invariance condition at the point $u_0$.
	\end{itemize}
Then $u_0\in K$ is a solution of the equation
\begin{equation}
D\Psi(u)=D\Phi(u).
\end{equation}
\end{thm}
In view of the above variational methods, in order to prove our existence results, one needs to apply Theorem \ref{key-theorem-in-the-variational-methods} in our context. For that, we proceed as follows.

Here and throughout the paper, $V=\cX^1_0\cap L^{p}(\Omega)$. Clearly, $V$ is a reflexive Banach space with the norm
\begin{equation*}
\|u\|_V:=\|u\|_{\cX^1_0}+\|u\|_{L^p(\Omega)}.
\end{equation*}
Let $K$ be the appropriate convex and weakly closed subset (to be specified later) of $V$. Define $\Phi: V\to \R$ as 
\begin{equation*}
\Phi(u)=\frac{1}{p}\int_{\Omega}|u|^p\ dx+\frac{\mu}{q}\int_{\Omega}|u|^q\ dx.
\end{equation*}
Clearly, $\Phi\in C^1(V,\R)$. Now, we consider the proper, convex, and lower semi-continuous function $\Psi: V\to\R$ defined as
\begin{equation*}
\Psi(u)=\frac{1}{2}\int_{\Omega}|\nabla u|^2\ dx+\frac{1}{2}\int_{\R^N}\int_{\R^N}\frac{(u(x)-u(y))^2}{|x-y|^{N+2s}}\ dxdy.
\end{equation*} 
Notice that $\Psi$ is G$\hat{\text{a}}$teaux differentiable with
\begin{equation*}
D\Psi(u)(v)=\langle u,v\rangle_{\cX^1_0}=\int_{\Omega}\nabla u\cdot\nabla v\ dx+\int_{\R^N}\int_{\R^N}\frac{(u(x)-u(y))(v(x)-v(y))}{|x-y|^{N+2s}}\ dxdy.
\end{equation*}
We now set $I(u)=\Psi(u)-\Phi(u)$ i.e.,
\begin{equation*}
I(u)=\frac{1}{2}\int_{\Omega}|\nabla u|^2\ dx+\frac{1}{2}\int_{\R^N}\int_{\R^N}\frac{(u(x)-u(y))^2}{|x-y|^{N+2s}}\ dxdy-\frac{1}{p}\int_{\Omega}|u|^p\ dx-\frac{\mu}{q}\int_{\Omega}|u|^q\ dx.
\end{equation*}
Then, $I$ is the Euler-Lagrange functional corresponding to \eqref{e2}. Denote by $I_K=\Psi_K-\Phi$ its restriction to $K$. We have the following.

\begin{thm}\label{t1}
	Let $V=\cX^1_0\cap L^p(\Omega)$, and let $K$ be a convex and weakly closed subset of $V$. Suppose the following two assertions hold:
	\begin{itemize}
		\item [$(i)$] The functional $I_K$ has a critical point $\tilde{u}\in V$ in the sense of Definition \ref{def3}, and;
		\item [$(ii)$] there exist $\tilde{v}\in K$ such that
		\begin{equation}\label{v-tilde-equation}
		-\Delta\tilde{v}+(-\Delta)^s\tilde{v}=D\Phi(\tilde{u})=|\tilde{u}|^{p-2}\tilde{u}+\mu|\tilde{u}|^{q-2}\tilde{u},
		\end{equation}
		in the weak sense i.e.,
		\begin{equation}\label{v-tilde-weak-formulation}
		\int_{\Omega}\nabla\tilde{v}\cdot\nabla\phi\ dx+\int_{\R^N}\int_{\R^N}\frac{(\tilde{v}(x)-\tilde{v}(y))(\phi(x)-\phi(y))}{|x-y|^{N+2s}}\ dxdy=\int_{\Omega}D\Phi(\tilde{u})\phi\ dx, \quad\forall\phi\in V.
		\end{equation}
	\end{itemize}
Then $\tilde{u}\in K$ is a weak solution of the equation
\begin{equation*}
-\Delta u+(-\Delta)^su=|u|^{p-2}u+\mu |u|^{q-2}u.
\end{equation*}
\end{thm}

\begin{proof}
	Since by assumption $(i)$ $\tilde{u}$ is a critical point of $I_K$, then by Definition \ref{def3}, we have
	\begin{align*}
	&\frac{1}{2}\int_{\Omega}|\nabla\phi|^2\ dx+\frac{1}{2}\int_{\R^N}\int_{\R^N}\frac{(\phi(x)-\phi(y))^2}{|x-y|^{N+2s}}\ dxdy-\frac{1}{2}\int_{\Omega}|\nabla\tilde{u}|^2\ dx-\frac{1}{2}\int_{\R^N}\int_{\R^N}\frac{(\tilde{u}(x)-\tilde{u}(y))^2}{|x-y|^{N+2s}}\ dxdy\\
	&\geq \langle D\Phi(\tilde{u}),\phi-\tilde{u}\rangle=:\int_{\Omega}D\Phi(\tilde{u})(\phi-\tilde{u})\ dx, \quad\quad\forall\phi\in K.
	\end{align*}
	By taking in particular $\phi=\tilde{v}$, the above inequality becomes
	\begin{align}\label{l1}
\nonumber	&\frac{1}{2}\int_{\Omega}|\nabla\tilde{v}|^2\ dx+\frac{1}{2}\int_{\R^N}\int_{\R^N}\frac{(\tilde{v}(x)-\tilde{v}(y))^2}{|x-y|^{N+2s}}\ dxdy-\frac{1}{2}\int_{\Omega}|\nabla\tilde{u}|^2\ dx-\frac{1}{2}\int_{\R^N}\int_{\R^N}\frac{(\tilde{u}(x)-\tilde{u}(y))^2}{|x-y|^{N+2s}}\ dxdy\\
	&\geq\int_{\Omega}D\Phi(\tilde{u})(\tilde{v}-\tilde{u})\ dx.
	\end{align}
	We now use $\phi=\tilde{v}-\tilde{u}$ as a test function in \eqref{v-tilde-weak-formulation} to get
	\begin{align}\label{l2}
	\int_{\Omega}\nabla\tilde{v}\cdot\nabla(\tilde{v}-\tilde{u})\ dx+\int_{\R^N}\int_{\R^N}\frac{(\tilde{v}(x)-\tilde{v}(y))((\tilde{v}-\tilde{u})(x)-(\tilde{v}-\tilde{u})(y))}{|x-y|^{N+2s}}\ dxdy=\int_{\Omega}D\Phi(\tilde{u})(\tilde{v}-\tilde{u})\ dx.
	\end{align}
	Plugging \eqref{l2} into \eqref{l1}, we get
	\begin{align}\label{l3}
	\nonumber	&\frac{1}{2}\int_{\Omega}|\nabla\tilde{v}|^2\ dx+\frac{1}{2}\int_{\R^N}\int_{\R^N}\frac{(\tilde{v}(x)-\tilde{v}(y))^2}{|x-y|^{N+2s}}\ dxdy-\frac{1}{2}\int_{\Omega}|\nabla\tilde{u}|^2\ dx-\frac{1}{2}\int_{\R^N}\int_{\R^N}\frac{(\tilde{u}(x)-\tilde{u}(y))^2}{|x-y|^{N+2s}}\ dxdy\\
	&\geq\int_{\Omega}\nabla\tilde{v}\cdot\nabla(\tilde{v}-\tilde{u})\ dx+\int_{\R^N}\int_{\R^N}\frac{(\tilde{v}(x)-\tilde{v}(y))((\tilde{v}-\tilde{u})(x)-(\tilde{v}-\tilde{u})(y))}{|x-y|^{N+2s}}\ dxdy.
	\end{align}
	Since $\Psi$ is a differentiable convex function, it follows that\footnote{This is a well-known property of convex functions: the graph of every differentiable convex function lies above all of its tangents.}
	\begin{equation*}
	\Psi(\tilde{u})\geq\Psi(\tilde{v})+D\Psi(\tilde{v})(\tilde{u}-\tilde{v})=\Psi(\tilde{v})+\langle \tilde{v},\tilde{u}-\tilde{v}\rangle_{\cX^1_0}
	\end{equation*}
	i.e.,
	\begin{align}\label{l4}
	\nonumber	&\frac{1}{2}\int_{\Omega}|\nabla\tilde{u}|^2\ dx+\frac{1}{2}\int_{\R^N}\int_{\R^N}\frac{(\tilde{u}(x)-\tilde{u}(y))^2}{|x-y|^{N+2s}}\ dxdy-\frac{1}{2}\int_{\Omega}|\nabla\tilde{v}|^2\ dx-\frac{1}{2}\int_{\R^N}\int_{\R^N}\frac{(\tilde{v}(x)-\tilde{v}(y))^2}{|x-y|^{N+2s}}\ dxdy\\
	&\geq\int_{\Omega}\nabla\tilde{v}\cdot\nabla(\tilde{u}-\tilde{v})\ dx+\int_{\R^N}\int_{\R^N}\frac{(\tilde{v}(x)-\tilde{v}(y))((\tilde{u}-\tilde{v})(x)-(\tilde{u}-\tilde{v})(y))}{|x-y|^{N+2s}}\ dxdy.
	\end{align}
	We now deduce from \eqref{l3} and \eqref{l4} that
	\begin{align}\label{l5}
	\nonumber&\frac{1}{2}\int_{\Omega}|\nabla\tilde{v}|^2\ dx+\frac{1}{2}\int_{\R^N}\int_{\R^N}\frac{(\tilde{v}(x)-\tilde{v}(y))^2}{|x-y|^{N+2s}}\ dxdy-\frac{1}{2}\int_{\Omega}|\nabla\tilde{u}|^2\ dx-\frac{1}{2}\int_{\R^N}\int_{\R^N}\frac{(\tilde{u}(x)-\tilde{u}(y))^2}{|x-y|^{N+2s}}\ dxdy\\
	&=\int_{\Omega}\nabla\tilde{v}\cdot\nabla(\tilde{v}-\tilde{u})\ dx+\int_{\R^N}\int_{\R^N}\frac{(\tilde{v}(x)-\tilde{v}(y))((\tilde{v}-\tilde{u})(x)-(\tilde{v}-\tilde{u})(y))}{|x-y|^{N+2s}}\ dxdy.
	\end{align}
	On the other hand, simple calculation yields
	\begin{align}\label{l6}
	\nonumber&\frac{1}{2}\int_{\Omega}|\nabla(\tilde{v}-\tilde{u})|^2\ dx+\frac{1}{2}\int_{\R^N}\int_{\R^N}\frac{((\tilde{v}-\tilde{u})(x)-(\tilde{v}-\tilde{u})(y))^2}{|x-y|^{N+2s}}\ dxdy\\
	\nonumber&=-\frac{1}{2}\int_{\Omega}|\nabla\tilde{v}|^2\ dx-\frac{1}{2}\int_{\R^N}\int_{\R^N}\frac{(\tilde{v}(x)-\tilde{v}(y))^2}{|x-y|^{N+2s}}\ dxdy+\frac{1}{2}\int_{\Omega}|\nabla\tilde{u}|^2\ dx+\frac{1}{2}\int_{\R^N}\int_{\R^N}\frac{(\tilde{u}(x)-\tilde{u}(y))^2}{|x-y|^{N+2s}} \\
	&+\int_{\Omega}\nabla\tilde{v}\cdot\nabla(\tilde{v}-\tilde{u})\ dx+\int_{\R^N}\int_{\R^N}\frac{(\tilde{v}(x)-\tilde{v}(y))((\tilde{v}-\tilde{u})(x)-(\tilde{v}-\tilde{u})(y))}{|x-y|^{N+2s}}\ dxdy.
	\end{align}
	Therefore, from \eqref{l5} and \eqref{l6} we deduce that
	\begin{equation*}
	\frac{1}{2}\int_{\Omega}|\nabla(\tilde{v}-\tilde{u})|^2\ dx+\frac{1}{2}\int_{\R^N}\int_{\R^N}\frac{((\tilde{v}-\tilde{u})(x)-(\tilde{v}-\tilde{u})(y))^2}{|x-y|^{N+2s}}\ dxdy=0
	\end{equation*}
	which implies that $\tilde{v}=\tilde{u}$ a.e., in $\R^N$. This completes the proof, thanks to \eqref{v-tilde-equation}.
\end{proof}
\begin{remark}\label{remark}
	We emphasize that assumption $(ii)$ of Theorem \ref{t1} says that the triple $(\Psi, K, \Phi)$ satisfies the pointwise invariance condition (at $\tilde{u}$) as in  Theorem \ref{key-theorem-in-the-variational-methods}. Notice that in fact from \eqref{v-tilde-equation}, we identify $G=0$ in the Definition \ref{def}.
\end{remark}

We now wish to establish our existence result stated in Theorems \ref{first-main-result} and \ref{second-main-result}. For this, as stated above, we apply Theorem \ref{key-theorem-in-the-variational-methods}. In order to achieve our goal, we first need to define the appropriate convex set $K$. For our purpose, we define the convex set $K$ as
\begin{equation}\label{convex-set}
K(r):=\{u\in V: \|u\|_{L^{\infty}(\Omega)}\leq r\}
\end{equation}
for some $r>0$ to be determined later. Then we have the following.

\begin{lemma}\label{lm1}
	Fix $r>0$. Then
	\begin{equation*}
	K(r)=\{u\in V: \|u\|_{L^{\infty}(\Omega)}\leq r\}
	\end{equation*}
	is a weakly closed subset of $V$.
\end{lemma}

\begin{proof}
	Let $u_i\in K(r)$ be a sequence such that $u_i\rightharpoonup u$ weakly in $V$. Then in particular, $u_i\rightharpoonup u$ weakly in $\cX^1_0$. Since the embedding $\cX^1_0\hookrightarrow L^2(\Omega)$ is compact (see \cite[Lemma 2.2]{su2022regularity}), then after passing to a subsequence, $u_i\rightarrow u$ strongly in $L^2(\Omega)$ and thus $u_i\rightarrow u$ a.e. in $\Omega$. Finally, since $\|u_i\|_{L^{\infty}(\Omega)}\leq r$, we deduce that $\|u\|_{L^{\infty}(\Omega)}\leq r$ and hence, $u\in K(r)$. This completes the proof.
\end{proof}
We also have the following.

\begin{lemma}\label{lm2}
	Suppose that
	\begin{equation*}
	K=\{u\in K(r): u(x)\geq0~~\text{for almost all}~~x\in\Omega\}.
	\end{equation*}
	Then there exists $\tilde{u}\in K$ such that $I_{K}(\tilde{u})=\inf_{u\in V}I_{K}(u)$.
\end{lemma}

\begin{proof}
	We set $\alpha=\inf_{u\in V}I_K(u)$. The aim is to show that the infimum $\alpha$ is achieved by a function $\tilde{u}\in K$. First of all, recalling the definition of $\Psi_K$, then $I_K=+\infty$ for $u\notin K$. We claim that $\alpha>-\infty$. In fact, for all $u\in K$,
	\begin{align*}
	\Phi(u)=\frac{1}{p}\int_{\Omega}|u|^p\ dx+\frac{\mu}{q}\int_{\Omega}|u|^q\ dx\leq c_1\|u\|^p_{L^{\infty}(\Omega)}+c_2\|u\|^q_{L^{\infty}(\Omega)}\leq c_1r^p+c_2r^q.
	\end{align*}
	Combining this with the fact that $\Psi(u)$ is nonnegative, we get
	\begin{equation*}
	I_K(u)=\Psi_K(u)-\Phi(u)\geq-(c_1r^p+c_2r^q)>-\infty
	\end{equation*} 
	which yields $\alpha>-\infty$. Now, let $\{u_i\}$ be a minimizing sequence of $I_K$ in $K$ i.e., $u_i\in K$ is such that $I_K(u_i)\rightarrow\alpha$. In particular, $\{I_K(u_i)\}$ is bounded and since likewise $\{\Phi(u_i)\}$, we also have that $\{\Psi_K(u_i)\}$ is bounded. Consequently, the sequence $\{u_i\}$ is bounded in $\cX^1_0$. Therefore, after passing to a subsequence, there is $\tilde{u}\in \cX^1_0$ such that $u_i\rightharpoonup \tilde{u}$ weakly in $\cX^1_0$, $u_i\rightarrow \tilde{u}$ strongly in $L^2(\Omega)$ thanks to the compact embedding $\cX^1_0\hookrightarrow L^2(\Omega)$ (see \cite[Lemma 2.2]{su2022regularity}). In particular, $u_i\rightarrow \tilde{u}$ a.e. in $\Omega$. Combining this with the fact that $\|u_i\|_{L^{\infty}(\Omega)}\leq r$, then $\|\tilde{u}\|_{L^{\infty}(\Omega)}\leq r$ and thus $\tilde{u}\in K$.
	
	We claim that $u_i\rightarrow \tilde{u}$ in $L^q(\Omega)$ and  $L^p(\Omega)$ respectively. In fact, since $q<2$, we have
	\begin{equation*}
	u_i\rightarrow \tilde{u}~~~\text{in}~~L^q(\Omega).
	\end{equation*}
	Regarding the $L^p$ convergence, we have
	\begin{align}\label{lp-convergence}
	\nonumber\|u_i-\tilde{u}\|^p_{L^p(\Omega)}&=\int_{\Omega}|u_i-\tilde{u}|^p\ dx=\int_{\Omega}|u_i-\tilde{u}|^{p-2}|u_i-u|^2\ dx\\
	\nonumber&\leq(\|u_i\|_{L^{\infty}(\Omega)}+\|\tilde{u}\|_{L^{\infty}(\Omega)})^{p-2}\|u_i-\tilde{u}\|^2_{L^2(\Omega)}\\
	&\leq(2r)^{p-2}\|u_i-\tilde{u}\|^2_{L^2(\Omega)}.
	\end{align}
	Letting $i\to\infty$ in \eqref{lp-convergence} it follows that $u_i\rightarrow \tilde{u}$ in $L^p(\Omega)$. As a consequence, we get $\Phi(u_i)\rightarrow\Phi(\tilde{u})$ as $i\to\infty$. Now, since by Proposition \ref{lower-semi-continuous-properti} $\Psi_K(\tilde{u})\leq\liminf_{i\to\infty}\Psi_K(u_i)$ we have
	\begin{align*}
	I_K(\tilde{u})=\Psi_K(\tilde{u})-\Phi(\tilde{u})\leq\liminf_{i\to\infty}(\Psi_K(u_i)-\Phi(u_i))=\liminf_{i\to\infty}I_K(u_i)=\alpha.
	\end{align*}
	The proof is therefore finished.
\end{proof}

The following lemma deals with the multiplicity of critical points of $I_K$.

\begin{lemma}\label{lm3}
	Suppose that $K=K(r)$. The functional $I_K$ admits infinitely many distinct critical points.
\end{lemma}

\begin{proof}
	First of all, it is clear that $I_K$ satisfies hypothesis \textbf{(H)} since $\Phi$ is continuously differentiable and $\Psi_K$ is a proper, convex, and lower semi-continuous function.
	 
	We now divide the rest of the proof into two steps. In the first step, we show that $I_K$ satisfies the \textbf{(PS)} condition and in the second, we establish the multiplicity result.\\
	
	\textbf{Step 1.} We show that $I_K$ satisfies the \textbf{(PS)} condition. Let $\{u_i\}$ be a sequence such that $I_K(u_i)\rightarrow c\in\R$ and
	\begin{align}\label{ps1}
	\langle D\Phi(u_i), u_i-v\rangle+\Psi_K(v)-\Psi_K(u_i)\geq-\varepsilon_i\|u_i-v\|_{V}, ~~\forall v\in V. 
	\end{align}
	Since $I_K(u_i)\rightarrow c$, then the sequence $\{I_K(u_i)\}$ is bounded. Since $\{\Phi(u_i)\}$ is also bounded, then $\{\Psi_K(u_i)\}$ is bounded. In particular, we have that $\{u_i\}$ is a bounded sequence in $\cX^1_0$. Hence, after passing to a subsequence, there exists $\tilde{u}\in\cX^1_0$ such that $u_i\rightharpoonup \tilde{u}$ weakly in $\cX^1_0$ and thanks to the compact embedding $\cX^1_0\hookrightarrow L^2(\Omega)$ (see \cite[Lemma 2.2]{su2022regularity}) we have $u_i\rightarrow\tilde{u}$ strongly in $L^2(\Omega)$. In particular, $u_i\rightarrow\tilde{u}$ a.e. in $\Omega$. As in the proof of Lemma \ref{lm2}, we have $\tilde{u}\in K$. Also, $u_i\rightarrow\tilde{u}$ in $L^p(\Omega)$ as $i\rightarrow\infty$. 
	
	Now,
	\begin{align}\label{ps2}
	\|u_i-\tilde{u}\|^2_{\cX^1_0}=\langle u_i-\tilde{u}, u_i-\tilde{u}\rangle_{\cX^1_0}=\langle u_i-\tilde{u}, u_i\rangle_{\cX^1_0}-\langle u_i-\tilde{u}, \tilde{u}\rangle_{\cX^1_0}.
	\end{align}
	Since $u_i\rightharpoonup \tilde{u}$ weakly in $\cX^1_0$ then
	\begin{equation}\label{ps3}
	\langle u_i-\tilde{u}, \tilde{u}\rangle_{\cX^1_0}\rightarrow0\quad\text{as}~~i\rightarrow\infty.
	\end{equation}
	On the other hand,
	\begin{align}\label{ps4}
	\nonumber\langle u_i-\tilde{u}, u_i\rangle_{\cX^1_0}&=\langle u_i, u_i\rangle_{\cX^1_0}-\langle\tilde{u}, u_i\rangle_{\cX^1_0}=2\Psi(u_i)-\langle\tilde{u}, u_i\rangle_{\cX^1_0}\\
	\nonumber&=2(\Psi(u_i)-\Psi(\tilde{u})-\langle D\Phi(u_i), u_i-\tilde{u}\rangle)\\
	&+2\Psi(\tilde{u})+2\langle D\Phi(u_i), u_i-\tilde{u}\rangle-\langle\tilde{u}, u_i\rangle_{\cX^1_0}. 
	\end{align}  
	Since $\{u_i-\tilde{u}\}$ is bounded in $V$, then it follows from \eqref{ps1} that
	\begin{equation}\label{ps5}
	\limsup_{i\rightarrow\infty}(\Psi(u_i)-\Psi(\tilde{u})-\langle D\Phi(u_i), u_i-\tilde{u}\rangle)\leq0.
	\end{equation}
	Moreover, using again that $u_i\rightharpoonup\tilde{u}$ weakly in $\cX^1_0$, then
	\begin{align}\label{ps6}
	2\Psi(\tilde{u})-\langle\tilde{u}, u_i\rangle_{\cX^1_0}=\langle\tilde{u}, \tilde{u}\rangle_{\cX^1_0}-\langle\tilde{u}, u_i\rangle_{\cX^1_0}=\langle\tilde{u}, \tilde{u}-u_i\rangle_{\cX^1_0}\rightarrow0\quad\text{as}~~i\rightarrow\infty.
	\end{align}
	Next, using that $u_i, \tilde{u}\in K$ then
	\begin{align}\label{ps7}
\nonumber|\langle D\Phi(u_i), u_i-\tilde{u}\rangle|&\leq\int_{\Omega}|u_i|^{p-1}|u_i-\tilde{u}|\ dx+\mu\int_{\Omega}|u_i|^{q-1}|u_i-\tilde{u}|\ dx\\
	&\leq|\Omega|^{\frac{1}{2}}(r^{p-1}+r^{q-1})\|u_i-\tilde{u}\|_{L^2(\Omega)}\rightarrow0\quad\quad\text{as}~~i\rightarrow\infty.
	\end{align}
	From \eqref{ps7}, \eqref{ps6} and \eqref{ps5}, it follows from \eqref{ps4} that	
  \begin{align}\label{ps8}
   \limsup_{i\rightarrow\infty}\langle u_i-\tilde{u}, u_i\rangle_{\cX^1_0}\leq0.
  \end{align}
	Combining \eqref{ps8} and \eqref{ps2} and recalling \eqref{ps1}, we deduce that $u_i\rightarrow \tilde{u}$ strongly in $\cX^1_0$. Hence, the functional $I_K$ satisfies the \textbf{(PS)} condition.\\
	
	\textbf{Step 2.} We prove the multiplicity result. For any $j\in\N$, let 
	\begin{equation*}
	c_j=\inf_{A\in Z_j}\sup_{u\in A}I(u).
	\end{equation*}
	Here, $Z_j$ is as in \eqref{a1}. In this step, we wish to apply Theorem \ref{multiplicity-result}. For that, we shall show that $c_j$ satisfies $-\infty<c_j<0$ for every $j\in\N$.
	Let $\lambda_j$ be the $j$-th eigenvalue of $-\Delta+(-\Delta)^s$ in $\cX^1_0$ (counted with its multiplicity) and consider $\xi_j$ an eigenfunction associated to $\lambda_j$. Then $\{\xi_j\}_{j\in\N}$ form an orthonormal (resp. orthogonal) basis of $L^2(\Omega)$ (resp. $\cX^1_0$) (see \cite[Proposition 2.4]{maione2022variational}). 
	
	Arguing as in Lemma \ref{lm2}, we have that $I_K$ is bounded from below. Hence, $c_j>-\infty$ for every $j\in\N$. It thus remains to prove that $c_j<0$. To this end, we set
	\begin{align*}
	A=\{u=\beta_1\xi_1+\cdots+\beta_j\xi_j: \|u\|^2_{\cX^1_0}=\beta_1^2+\cdots+\beta_j^2=\rho^2\}
	\end{align*}
	 for some $\rho>0$ to be determined later. From Proposition \ref{genus}, $\gamma(A)=j$. Therefore, $A\in Z_j$. Since $A$ is a subset of a finite dimensional space and that norms are equivalent in finite dimensional space, then we choose $\rho$ sufficiently small such that $A\subseteq K(r)$. Moreover, one can find $a_1, a_2>0$ such that $\|u\|_{L^p(\Omega)}> a_1\|u\|_{\cX^1_0}$ and $\|u\|_{L^q(\Omega)}> a_2\|u\|_{\cX^1_0}$ for all $u\in A$. Thus,
	 \begin{align*}
	 I_K(u)&=\frac{1}{2}\|u\|^2_{\cX^1_0}-\frac{1}{p}\|u\|^p_{L^p(\Omega)}-\frac{\mu}{q}\|u\|^q_{L^q(\Omega)}\\
	 &<\frac{1}{2}\rho^2-\frac{1}{p}a_1^p\rho^p-\frac{\mu}{q}a_2^q\rho^q=\rho^q\Big(\frac{1}{2}\rho^{2-q}-\frac{1}{p}a^p_1\rho^{p-q}-\frac{\mu}{q}a^q_2\Big).
	 \end{align*}
	 We now choose $\rho$ sufficiently small such that $I_K(u)<\rho^q\Big(\frac{1}{2}\rho^{2-q}-\frac{1}{p}a^p_1\rho^{p-q}-\frac{\mu}{q}a^q_2\Big)<0$ for any $u\in A$. This implies that $c_j<0$ for any $j\in\N$. Hence, by Theorem \ref{multiplicity-result}, we deduce that $I_K$ admits a sequence of dinstinct critical points $\{u_i\}_{i\in\N}$ in the sense of Definition \ref{def3}. The proof is therefore finished by  \textbf{Steps 1} and \textbf{2}.
\end{proof}

The next lemmas show that pointwise assumption $(ii)$ of Theorem \ref{t1} is satisfied in our context.

\begin{lemma}\label{lm4}
	Let $\Omega$ be a bounded domain in $\R^N$ and assume that $1<q<2<p$. Then
	\begin{equation*}
	\|D\Phi(u)\|_{L^{\infty}(\Omega)}\leq r^{p-1}+\mu r^{q-1},\quad\quad\forall u\in K(r).
	\end{equation*}
\end{lemma}

\begin{proof}
	The proof of this lemma is straightforward. We give details for the sake of completeness. Let $u\in K(r)$. We have
	\begin{align*}
	\|D\Phi(u)\|_{L^{\infty}(\Omega)}&=\|u|u|^{p-2}+\mu u|u|^{p-2}\|_{L^{\infty}(\Omega)}\\
	&\leq\|u\|^{p-1}_{L^{\infty}(\Omega)}+\mu\|u\|^{q-1}_{L^{\infty}(\Omega)}\\
	&\leq r^{p-1}+\mu r^{q-1}.
	\end{align*}
\end{proof}

\begin{lemma}\label{lm5}
	If $1<q<2<p$, there exists $\mu_\#>0$ such that for every $\mu\in(0,\mu_\#)$, there exist $r_1, r_2\in\R$ with $r_1<r_2$ such that for every $r\in [r_1, r_2]$ and every $\tilde{u}\in K(r)$, the equation
	\begin{equation}
	-\Delta v+(-\Delta)^sv=|\tilde{u}|^{p-2}\tilde{u}+\mu|\tilde{u}|^{q-2}\tilde{u}
	\end{equation}
	admits a weak solution $v\in K(r)$. 
\end{lemma}

\begin{proof}
	Let $\tilde{u}\in K(r)$. Then $\|\tilde{u}\|_{L^{\infty}(\Omega)}<r$ and thus $\tilde{u}\in L^{\infty}(\Omega)$. Hence, by Lemma \ref{l-infty-estimate}, there exists $v\in\cX^1_0$ satisfying in the weak sense the equation
	\begin{equation*}
	-\Delta v+(-\Delta)^sv=|\tilde{u}|^{p-2}\tilde{u}+\mu |\tilde{u}|^{q-2}\tilde{u}.
	\end{equation*}
	Moreover, $v\in L^{\infty}(\Omega)$ and there exists a positive constant $C>0$ such that
	\begin{align}\label{a2}
	\nonumber\|v\|_{L^{\infty}(\Omega)}&\leq C\||\tilde{u}|^{p-2}\tilde{u}+\mu|\tilde{u}|^{q-2}\tilde{u}\|_{L^{\infty}(\Omega)}\\
	&=C\|D\Phi(\tilde{u})\|_{L^{\infty}(\Omega)}\leq C(r^{p-1}+\mu r^{q-1}).
	\end{align}
	Note that in \eqref{a2} we used Lemma \ref{lm4}. Therefore $v\in K(r)$ if and only if
	\begin{equation}
	C(r^{p-1}+\mu r^{q-1})\leq r\quad i.e.,~~ C(r^{p-2}+\mu r^{q-2})\leq 1.
	\end{equation}
	By studying the variation of the function $h(r)=C(r^{p-2}+\mu r^{q-2})-1$, it follows that there exists $\mu_\#>0$ such that for every $\mu\in(0,\mu_\#)$, there exist $r_1<r_2$ such that for all $r\in [r_1, r_2]$, we have $h(r)\leq0$. This completes the proof.  
\end{proof}

Having the above results in mind, we are now in position to prove Theorem \ref{first-main-result}.

\begin{proof}[Proof Theorem \ref{first-main-result}]
	Let $r_1, r_2$ and $\mu_\#$ be as in Lemma \ref{lm5}. Define
	\begin{equation}
	K:=\{u\in K(r): u(x)\geq0~~\text{a.e.}~~x\in\Omega\}
	\end{equation}
	for some $r\in(r_1, r_2)$. By Lemma \ref{lm2}, there exists $\tilde{u}\in K$ such that $I_K(\tilde{u})=\inf_{u\in V}I_K(u)$. Moreover, $\tilde{u}$ is a critical point of $I_K$ in the sense of Definition \ref{def3}. 
	In fact, since $\tilde{u}$ is a minimizer of $I_K$ in $V$, then
	\begin{align*}
	I_K(\tilde{u})\leq I_K((1-t)\tilde{u}+tv)\quad\quad \forall v\in V
	\end{align*} 
	for all sufficiently small $t>0$.
	Recalling that $I_K=\Psi_K-\Phi$, then from the inequality above and using that $\Psi_K$ is convex, we have
	\begin{align*}
	0\leq I_K((1-t)\tilde{u}+tv)-I_K(\tilde{u})&=\Phi(\tilde{u})-\Phi((1-t)\tilde{u}+tv)+\Psi_K((1-t)\tilde{u}+tv)-\Psi_K(\tilde{u})\\
	&\leq \Phi(\tilde{u})-\Phi(\tilde{u}+t(v-\tilde{u}))+t(\Psi_K(v)-\Psi_K(\tilde{u})).
	\end{align*}
	The identity \eqref{critical-point} then follows by dividing the inequality above by $t$ and letting $t\rightarrow0$.
	
	On the other hand, by Lemma \ref{lm5}, there exists $v\in K(r)$ solving in weak sense the equation $-\Delta v+(-\Delta)^sv=|\tilde{u}|^{p-2}\tilde{u}+\mu|\tilde{u}|^{q-2}\tilde{u}$. Moreover, since $\tilde{u}\in K$ (and thus $\tilde{u}(x)\geq0$ a.e. $x\in\Omega$) then $-\Delta v+(-\Delta)^sv=|\tilde{u}|^{p-2}\tilde{u}+\mu|\tilde{u}|^{q-2}\tilde{u}\geq0$. Hence, by Proposition \ref{maximum-principle}, we deduce that $v\geq0$ a.e. in $\Omega$ and therefore $v\in K$. By Theorem \ref{t1} we also deduce that $\tilde{u}$ is a weak solution of \eqref{e2} i.e.,

   \begin{equation}\label{u-tilde-weak-formulation}
		\int_{\Omega}\nabla\tilde{u}\cdot\nabla\phi\ dx+\int_{\R^N}\int_{\R^N}\frac{(\tilde{u}(x)-\tilde{u}(y))(\phi(x)-\phi(y))}{|x-y|^{N+2s}}\ dxdy=\int_{\Omega}(|\Tilde{u}|^{p-2}\tilde{u}+\mu |\Tilde{u}|^{q-2}\Tilde{u})\phi\ dx, \quad\forall\phi\in V.
		\end{equation}
 To complete the proof, it remains to show that $I_K(\tilde{u})=\inf_{u\in V}I_{}(u)<0$ and $\tilde{u}>0$. Let us start by proving that $I_K(\tilde{u})=\inf_{u\in V}I_{K}(u)<0$. For this, let $u_0\in K$. Then for all $t\in[0,1]$, $tu_0\in K$. Now
	\begin{align*}
	&I_{K}(tu_0)\\
	&=\frac{1}{2}\int_{\Omega}|\nabla(tu_0)|^2\ dx+\frac{1}{2}\int_{\R^N}\int_{\R^N}\frac{(tu_0(x)-tu_0(y))^2}{|x-y|^{N+2s}}\ dxdy-\frac{1}{p}\int_{\Omega}|tu_0|^p\ dx-\frac{\mu}{q}\int_{\Omega}|tu_0|^q\ dx\\
	&=t^q\Bigg(\frac{t^{2-q}}{2}\int_{\Omega}|\nabla u_0|^2\ dx+\frac{t^{2-q}}{2}\int_{\R^N}\int_{\R^N}\frac{(u_0(x)-u_0(y))^2}{|x-y|^{N+2s}}\ dxdy-\frac{t^{p-q}}{p}\int_{\Omega}|u_0|^p\ dx-\frac{\mu}{q}\int_{\Omega}|u_0|^q\ dx\Bigg).
	\end{align*}
	For $t$ sufficiently small, we have from the above equality that $I_K(tu_0)<0$ and thus $I_K(\tilde{u})=\inf_{u\in V} I_K(u)<0$. Notice that a direct consequence of this is that $\tilde{u}$ is a nontrivial solution. 
	
	Now we show that $\tilde{u}>0$. First of all, we know that $\tilde{u}\geq0$ in $\R^N$. We now argue by contradiction. Assume that there exists $x_0\in\Omega$ and $R>0$ such that
   \begin{equation}\label{x}
       \tilde{u}\equiv0 \quad\text{a.e. on}~~B_R(x_0)\subset\subset\Omega.
   \end{equation}
Let $\phi\in C^{\infty}_c(\Omega)$ be a non-negative function with the properties
\begin{equation*}
     \supp(\phi)\subseteq B_R(x_0) \quad\quad \text{and}\quad\quad \int_{B_R(x_0)}\phi\ dx=1.
 \end{equation*}
Using $\phi$ as a test function in \eqref{u-tilde-weak-formulation} and from \eqref{x}, we have
\begin{align}\label{x1}
 \nonumber 0=\int_{B_R(x_0)}(|\Tilde{u}|^{p-2}\tilde{u}&+\mu |\Tilde{u}|^{q-2}\Tilde{u})\phi\ dx=\int_{\Omega}(|\Tilde{u}|^{p-2}\tilde{u}+\mu |\Tilde{u}|^{q-2}\Tilde{u})\phi\ dx\\
 \nonumber &=\int_{\Omega}\nabla\tilde{u}\cdot\nabla\phi\ dx+\int_{\R^N}\int_{\R^N}\frac{(\tilde{u}(x)-\tilde{u}(y))(\phi(x)-\phi(y))}{|x-y|^{N+2s}}\ dxdy\\
 \nonumber &=\int_{\R^N}\int_{\R^N}\frac{(\tilde{u}(x)-\tilde{u}(y))(\phi(x)-\phi(y))}{|x-y|^{N+2s}}\ dxdy\\
 \nonumber &=-2\int_{\Omega\setminus B_R(x_0)}\int_{B_R(x_0)}\frac{\tilde {u}(y)\phi(x)}{|x-y|^{N+2s}}\ dxdy\\
    &\leq-\frac{2}{\diam(\Omega)^{N+2s}}\int_{\Omega\setminus B_R(x_0)}\tilde{u}(y)\ dy.
\end{align}
 Recalling that $\tilde{u}\geq0$ in $\R^N$, we deduce from \eqref{x1} that $\tilde{u}\equiv0$ a.e. in $\Omega\setminus B_R(x_0)$ and thus, owing to \eqref{x}, we have $\Tilde{u}\equiv0$ a.e. in $\Omega$. This contradicts the fact that $\tilde{u}$ is  a nontrivial solution. Hence, $\tilde{u}>0$ in $\Omega$, as desired.
\end{proof}

Next, we give the proof of Theorem \ref{second-main-result}.

\begin{proof}[Proof of Theorem \ref{second-main-result}]
	This is a direct consequence of Lemmas \ref{lm5}, \ref{lm3} and Theorem \ref{t1}. In fact, consider $\mu_\#, r_1$ and $r_2$ be as in Lemma \ref{lm5}. Take $K=K(r_2)$. Then from Lemma \ref{lm3}. the functional $I_K$ possesses infinitely many distinct critical points. Moreover, again by Lemma \ref{lm5}, for any critical point $u_i$ of $I_K$, there exists $v_i\in K$ satisfying $-\Delta v_i+(-\Delta)^sv_i=|u_i|^{p-2}u_i+\mu |u_i|^{q-2}u_i$. Now, thanks to Theorem \ref{t1}, $\{u_i\}_{i\in\N}$ is a sequence of distinct solutions of \eqref{e2} such that $I_K(u_i)<0$ for all $i\in\N$.
\end{proof}
We now wish to prove Theorem \ref{third-main-result}. Before, we set
\begin{equation}
\cK=\{u\in K(r): u(\sigma_j(x))=-u(x)~~~\text{for all}~~x\in\Omega\}
\end{equation}
where $K(r)$ is defined as in \eqref{convex-set} and $\sigma_j$ the reflection defined in \eqref{reflection}. Also, $V$ and $I_{\cK}$ are defined similarly as above. We now give the proof of Theorem \ref{third-main-result}. 

\begin{proof}[Proof of Theorem \ref{third-main-result}]
	It suffices to prove that assumptions $(i)$ and $(ii)$ of Theorem \ref{t1} are fulfilled.\\
	
	\textbf{Step 1.} Proof of assumption $(i)$. We first observe that since $\cK\subset K(r)$, then $\inf_{u\in V}I_{\cK}(u)\geq\inf_{u\in V}I_{K(r)}(u)>-\infty$. Thus, one can choose a sequence $\{u_i\}$ in $\cK$ such that $I_{\cK}(u_i)\rightarrow\inf_{u\in V}I_{\cK}(u)$. In particular, $\{I_{\cK}(u_i)\}$ is bounded and since $\{\Phi(u_i)\}$ is also bounded, then one has that $\{\Psi_{\cK}(u_i)\}$ is bounded as well. This means that $\{u_i\}$ is a bounded sequence in $\cX^1_0$. Therefore, up to a subsequence, there exists $\tilde{u}\in \cX^1_0$ such that $u_i\rightharpoonup\tilde{u}$ weakly in $\cX^1_0$ and $u_i\rightarrow\tilde{u}$ strongly in $L^2(\Omega)$ (thanks to the compact embedding $\cX^1_0\hookrightarrow L^2(\Omega)$, see \cite[Lemma 2.2]{su2022regularity}). We also have in particular that $u_i\rightarrow\tilde{u}$ a.e. in $\Omega$. Moreover, since $\cK$ is weakly closed in $V$, then $\tilde{u}\in \cK$. Now, using the fact that $\Phi$ is continuous and $\Psi_{\cK}$ lower semi-continuous, we find that $I_{\cK}(\tilde{u})\leq\inf_{u\in V}I_{\cK}(u)$. This means that $\tilde{u}$ is a minimizer (and thus a critical point) of $I_{\cK}$ in $V$.\\
	
	\textbf{Step 2.} Proof of assumption $(ii)$. Since $\tilde{u}\in K(r)$, then by Lemma \ref{lm5}, there exists $v\in K(r)$ satisfying in the weak sense the equation
	\begin{equation}\label{c}
	-\Delta v+(-\Delta)^sv=|\tilde{u}|^{p-2}\tilde{u}+\mu|\tilde{u}|^{q-2}\tilde{u}.
	\end{equation}
	To complete the proof, we have to show that $v\in \cK$. Set $w(x)=v(\sigma_j(x))+v(x)$ for all $x\in\Omega$. Then from \eqref{c} and recalling that $\tilde{u}(\sigma_j(x))=-\tilde{u}(x)$, we have
	\begin{align*}
	-\Delta w(x)+(-\Delta)^sw(x)&=-\Delta v(\sigma_j(x))+(-\Delta)^sv(\sigma_j(x))+-\Delta v(x)+(-\Delta)^sv(x)\\
	 &=|\tilde{u}(\sigma_j(x))|^{p-2}\tilde{u}(\sigma_j(x))+\mu|\tilde{u}(\sigma_j(x))|^{q-2}\tilde{u}(\sigma_j(x))\\
	 &~~~~~+|\tilde{u}(x)|^{p-2}\tilde{u}(x)+\mu|\tilde{u}(x)|^{q-2}\tilde{u}(x)\\
	 &=0.
	\end{align*}
	Moreover, since also $w(x)=0$ for all $x\in\R^N\setminus\Omega$, then from maximum principle (see Proposition \ref{maximum-principle}) we deduce that $w(x)=0$ i.e., $v(\sigma_j(x))=-v(x)$ for all $x\in\Omega$. Consequently, $v\in \cK$, as desired. 
	We finally notice that arguing as in the proof of Theorem \ref{first-main-result}, we have that $\tilde{u}$ is a nontrivial solution of \eqref{e2}. 
\end{proof}

We wish now to prove Theorem \ref{non-existence-result}. To this end, we need some preliminary lemmas. We start with the following standard result which says one can find a solution between sub and super solutions. It reads as follows.

\begin{lemma}\label{caparison-principle}
	Assume that there exists a subsolution $u_1$ and a supersolution $u_2$ to problem \eqref{e2} with $u_1\leq u_2$. Then there also exists a solution $\tilde{u}$ with $u_1\leq\tilde{u}\leq u_2$.
\end{lemma}

\begin{proof}
	The proof of this lemma is classic and we omit it.
\end{proof}
We also need the following Brezis-Kamin-Oswald type result for the operator $L=-\Delta+(-\Delta)^s$.

\begin{lemma}\label{b-k-o-result}
	Assume that the function $\frac{f(\tau)}{\tau}$ is decreasing for $\tau>0$ and let $u_1, u_2\in V$ be respectively positive subsolution and supersolution of the problem
	\begin{equation}\label{t}
	\left\{\begin{aligned}
	Lu&=f(u)\quad\text{in}~~\Omega\\
	u&=0\quad\quad~\text{in}~~\R^N\setminus\Omega.
	\end{aligned}
	\right.
	\end{equation}
	Then $u_1\leq u_2$ in $\R^N$.
\end{lemma}

To prove the above lemma, the following mixed local and nonlocal Picone inequality plays a key role. 

\begin{prop}[Picone inequality]\label{picone-inequality}
	If $u, v\in V$, $Lu$ is a positive bounded Radon measure in $\Omega$ and $u\geq0$ and not identically zero, then 
	\begin{equation}\label{pic-in}
	\int_{\Omega}\frac{Lu}{u}v^2\ dx\leq\|v\|^2_{\cX^1_0}.
	\end{equation} 
\end{prop}

\begin{proof}
	The proof of this proposition is very similar to that of \cite[Theorem 18]{leonori2015basic}. For the sake of completeness, we give details henceforth. 
	
	For all $k\geq0$, we introduce the function $T_k: \R^+\rightarrow\R$ define by $T_k(t)=\max\{-k,\min\{k,t\}\}$. We set $v_k=T_k(v)$. Now, for every $\tau>0$, we also set $\overline{u}=u+\tau$. Then, simple calculation yields $\frac{v^2_k}{\overline{u}}\in \cX^1_0$. Now,
	\begin{align}\label{d}
	\int_{\Omega}Lu\frac{v^2_k}{\overline{u}}\ dx=\int_{\Omega}L\overline{u}\frac{v^2_k}{\overline{u}}\ dx=\int_{\Omega}\nabla\overline{u}\cdot\nabla\Big(\frac{v^2_k}{\overline{u}}\Big)\ dx+\int_{\R^N}\int_{\R^N}\frac{(\overline{u}(x)-\overline{u}(y))\Big(\frac{v_k(x)^2}{\overline{u}(x)}-\frac{v_k(y)^2}{\overline{u}(y)}\Big)}{|x-y|^{N+2s}}\ dxdy.
	\end{align}
	% is a direct consequence of classical Picone inequality (see \cite[Theorem 1.1]{allegretto1998picone}) combined with a nonlocal Picone inequality (see \cite[Proposition 4.2]{brasco2014convexity}).
	 From the classical Picone inequality (see \cite[Theorem 1.1]{allegretto1998picone}) it holds that
	\begin{equation}\label{classical-picone-inequality}
	\nabla\overline{u}\cdot \nabla\Big(\frac{v_k^2}{\overline{u}}\Big)\leq |\nabla v_k|^2.
	\end{equation}
	Integrating \eqref{classical-picone-inequality} over $\Omega$, we get that
	\begin{equation}\label{pic-in-classical}
	\int_{\Omega}\nabla\overline{u}\cdot \nabla\Big(\frac{v_k^2}{\overline{u}}\Big)\ dx\leq\int_{\Omega}|\nabla v_k|^2\ dx=\|\nabla v_k\|^2_{L^2(\Omega)}.
	\end{equation}
	On the other hand, from the ``nonlocal'' Picone inequality (see \cite[Proposition 4.2]{brasco2014convexity}), we have
	\begin{equation}\label{nonlocal-picone-inequality}
	(\overline{u}(x)-\overline{u}(y))\Big(\frac{v_k(x)^2}{\overline{u}(x)}-\frac{v_k(y)^2}{\overline{u}(y)}\Big)\leq |v_k(x)-v_k(y)|^2. 
	\end{equation}
	Now, multiplying \eqref{nonlocal-picone-inequality} by $|x-y|^{-N-2s}$ and integrating over $\R^N\times\R^N$, we obtain
	\begin{equation}\label{pic-in-nonlocal}
	\int_{\R^N}\int_{\R^N}\frac{(\overline{u}(x)-\overline{u}(y))\Big(\frac{v_k(x)^2}{\overline{u}(x)}-\frac{v_k(y)^2}{\overline{u}(y)}\Big)}{|x-y|^{N+2s}}\ dxdy\leq \int_{\R^N}\int_{\R^N}\frac{|v_k(x)-v_k(y)|^2}{|x-y|^{N+2s}}\ dxdy=[v_k]^2_s.
	\end{equation}
	Summing up \eqref{pic-in-classical} and \eqref{pic-in-nonlocal}, it follows from \eqref{d} that
	\begin{equation}\label{r}
	\int_{\Omega}Lu\frac{v^2_k}{\overline{u}}\ dx\leq \|v_k\|^2_{L^2(\Omega)}+[v_k]^2_s=\|v_k\|^2_{\cX^1_0}=\|T_k(v)\|^2_{\cX^1_0}.
	\end{equation}
	Now, \eqref{pic-in} follows from \eqref{r} by using Fatou Lemma together with monotone convergence theorem and the fact that $\tau$ was chosen arbitrarily.
\end{proof}

We now prove the Brezis-Kamin-Oswald type result. We follow lines of the proof of Theorem 20 in \cite{leonori2015basic}.

\begin{proof}[Proof of Lemma \ref{b-k-o-result}]
	We want to prove that the set $\Pi=\{x\in \Omega: u_1(x)>u_2(x)\}$ has zero measure i.e., $|\Pi|=0$. Define $w=(u^2_1-u^2_2)_+$. From the Picone inequality \eqref{pic-in} we deduce that $\frac{w}{u_1}$ and $\frac{w}{u_2}$ are admissible test functions to problem \eqref{t}. Therefore,
	\begin{align}\label{f}
	\int_{\Omega}u_1L\Big(\frac{w}{u_1}\Big)\ dx-\int_{\Omega}u_2L\Big(\frac{w}{u_2}\Big)\ dx\leq\int_{\Omega}\Big(\frac{f(u_1)}{u_1}-\frac{f(u_2)}{u_2}\Big)w\ dx.
	\end{align}
	Since $\tau\mapsto\frac{f(\tau)}{\tau}$ is decreasing by assumption, then the right-hand side of \eqref{f} is negative. The proof is completed if we show that the left-hand side in \eqref{f} is nonnegative. In fact, if this was the case, then one would have that $w\equiv0$ and thus $u_1\leq u_2$.
	
	Let us now show that the left-hand side in \eqref{f} is nonnegative. For that, we set $\overline{u}_1=u_1\chi_{\Pi},~\overline{u}_2=u_2\chi_{\Pi}$ and $w=\overline{u}^2_1-\overline{u}_2^2$. It then suffices to prove that
	\begin{align}\label{y}
	\nonumber\nabla &u_2(x)\cdot\nabla\Big(\frac{\overline{u}_1(x)^2-\overline{u}_2(x)^2}{u_2(x)}\Big)+(u_2(x)-u_2(y))\Big(\frac{\overline{u}_1(x)^2-\overline{u}_2(x)^2}{u_2(x)}-\frac{\overline{u}_1(y)^2-\overline{u}_2(y)^2}{u_2(y)}\Big)\\
	&\leq \nabla u_1(x)\cdot\nabla\Big(\frac{\overline{u}_1(x)^2-\overline{u}_2(x)^2}{u_1(x)}\Big)+(u_1(x)-u_1(y))\Big(\frac{\overline{u}_1(x)^2-\overline{u}_2(x)^2}{u_1(x)}-\frac{\overline{u}_1(y)^2-\overline{u}_2(y)^2}{u_1(y)}\Big)
	\end{align}
	i.e.,
	\begin{align}\label{g}
	\nonumber&\nabla u_2(x)\cdot\nabla\Big(\frac{\overline{u}_1(x)^2}{u_2(x)}\Big)+(u_2(x)-u_2(y))\Big(\frac{\overline{u}_1(x)^2}{u_2(x)}-\frac{\overline{u}_1(y)^2}{u_2(y)}\Big)\\
	\nonumber&~~~~+\nabla u_1(x)\cdot\nabla\Big(\frac{\overline{u}_2(x)^2}{u_1(x)}\Big)+(u_1(x)-u_1(y))\Big(\frac{\overline{u}_2(x)^2}{u_1(x)}-\frac{\overline{u}_2(y)^2}{u_1(y)}\Big)\\
	\nonumber&\leq \nabla u_2(x)\cdot\nabla\overline{u}_2(x)+(u_2(x)-u_2(y))(\overline{u}_2(x)-\overline{u}_2(y))\\
	&~~~~+\nabla u_1(x)\cdot\nabla\overline{u}_1(x)+(u_1(x)-u_1(y))(\overline{u}_1(x)-\overline{u}_1(y)).
	\end{align}
	Set
	\begin{align*}
	&Q_1=\nabla u_2(x)\cdot\nabla\Big(\frac{\overline{u}_1(x)^2}{u_2(x)}\Big)+(u_2(x)-u_2(y))\Big(\frac{\overline{u}_1(x)^2}{u_2(x)}-\frac{\overline{u}_1(y)^2}{u_2(y)}\Big),\\
	&Q_2=\nabla u_1(x)\cdot\nabla\Big(\frac{\overline{u}_2(x)^2}{u_1(x)}\Big)+(u_1(x)-u_1(y))\Big(\frac{\overline{u}_2(x)^2}{u_1(x)}-\frac{\overline{u}_2(y)^2}{u_1(y)}\Big),\\
	&Q_3= \nabla u_2(x)\cdot\nabla\overline{u}_2(x)+(u_2(x)-u_2(y))(\overline{u}_2(x)-\overline{u}_2(y)),\\
	&Q_4=\nabla u_1(x)\cdot\nabla\overline{u}_1(x)+(u_1(x)-u_1(y))(\overline{u}_1(x)-\overline{u}_1(y)),
	\end{align*}
 so that \eqref{g} is equivalent to
 \begin{equation}\label{sum}
 Q_1+Q_2\leq Q_3+Q_4.
 \end{equation}
	We now distinguish the following cases: $x, y\in\Pi$, $x, y\notin\Pi$, $x\in \Pi, y\notin\Pi$ and $x\notin\Pi, y\in \Pi$. \\
	
	\textbf{Case 1:} $x, y\in\Pi$. In this case, thanks to Picone inequality, we have $Q_1\leq Q_3$ and $Q_2\leq Q_4$. Hence, \eqref{g} (and thus \eqref{y}) holds true. \\
	
	\textbf{Case 2:} $x, y\notin\Pi$. In this case, we have from the definition of $\overline{u}_1$ and $\overline{u}_2$ that $Q_1=Q_2=Q_3=Q_4=0$.  Hence, \eqref{g} (and thus \eqref{y}) holds true.\\
	
	\textbf{Case 3:} $x\in \Pi, y\notin\Pi$. In this case, recalling again the definition of $\overline{u}_1$ and $\overline{u}_2$ we have
	\begin{align*}
	Q_1+Q_2&=\nabla u_2(x)\cdot\nabla\Big(\frac{\overline{u}_1(x)^2}{u_2(x)}\Big)+\overline{u}_1(x)^2-\overline{u}_1(x)^2\frac{u_2(y)}{u_2(x)}\\
	&~~+\nabla u_1(x)\cdot\nabla\Big(\frac{\overline{u}_2(x)^2}{u_1(x)}\Big)+\overline{u}_2(x)^2-\overline{u}_2(x)^2\frac{u_1(y)}{u_1(x)}
	\end{align*}
	and
	\begin{align*}
	Q_3+Q_4&=\nabla u_2(x)\cdot\nabla\overline{u}_2(x)+\overline{u}_2(x)^2-\overline{u}_2(x)u_2(y)\\
	&~~+\nabla u_1(x)\cdot\nabla\overline{u}_1(x)+\overline{u}_1(x)^2-\overline{u}_1(x)u_1(y).
	\end{align*}
	Hence, \eqref{sum} is equivalent to
	\begin{align}\label{k}
	\nonumber&u_1(y)\Big(\overline{u}_1(x)-\frac{\overline{u}_2(x)^2}{u_1(x)}\Big)+\nabla u_2(x)\cdot\nabla\Big(\frac{\overline{u}_1(x)^2}{u_2(x)}\Big)+\nabla u_1(x)\cdot\nabla\Big(\frac{\overline{u}_2(x)^2}{u_1(x)}\Big)\\
	&\leq u_2(y)\Big(\frac{\overline{u}_1(x)^2}{u_2(x)}-\overline{u}_2(x)\Big)+|\nabla\overline{u}_1(x)|^2+|\nabla\overline{u}_2(x)|^2.
	\end{align}
	But, since from the classical Picone inequality (see \cite[Theorem 1.1]{allegretto1998picone})
	\begin{equation*}
	\nabla u_2(x)\cdot\nabla\Big(\frac{\overline{u}_1(x)^2}{u_2(x)}\Big)\leq|\nabla\overline{u}_1(x)|^2 \quad\quad\text{and} \quad\quad\nabla u_1(x)\cdot\nabla\Big(\frac{\overline{u}_2(x)^2}{u_1(x)}\Big)\leq |\nabla\overline{u}_2(x)|^2,
	\end{equation*}
	then it suffices to show that
	\begin{equation*}
	u_1(y)\Big(\overline{u}_1(x)-\frac{\overline{u}_2(x)^2}{u_1(x)}\Big)\leq u_2(y)\Big(\frac{\overline{u}_1(x)^2}{u_2(x)}-\overline{u}_2(x)\Big)
	\end{equation*}
	which is also equivalent to
	\begin{equation*}
	\frac{u_1(y)}{\overline{u}_1(x)}\leq \frac{u_2(y)}{\overline{u}_2(x)} \quad\quad i.e.,~~u_1(y)\overline{u}_2(x)\leq u_2(y)\overline{u}_1(x).
	\end{equation*}
	The above inequality remains true since $x\in \Pi$ and $y\notin\Pi$. Therefore,  \eqref{g} (and thus \eqref{y}) also holds true. \\
	
	\textbf{Case 4:} $x\notin\Pi, y\in \Pi$. This case follows analogously to  \textbf{Case 3} by interchanging the roles of $x$ and $y$.\\
	
	In conclusion, we have just seen that in all cases, the inequality \eqref{y} is always true, as wanted. The proof is therefore finished.	
\end{proof}
Having the above preliminary results in mind, we are now ready to prove Theorem \ref{non-existence-result}.

\begin{proof}[Proof of Theorem \ref{non-existence-result}]
	We set 
	\begin{equation*}
	\Lambda=\sup\{\mu>0: \eqref{e2}_{\mu}~~\text{has a solution}\}.
	\end{equation*}
	Then, the constant $\mu_\#$ of Theorem \ref{first-main-result} satisfies $0<\mu_\#\leq\Lambda$. Now, consider $\tilde{\lambda}$ with the property
	\begin{equation}\label{m1}
	\tau^{p-1}+\tilde{\lambda}\tau^{q-1}>\lambda_1\tau,\quad\quad\forall\tau>0
	\end{equation}
	where $\lambda_1$ is the first Dirichlet eigenvalue of $-\Delta+(-\Delta)^s$ (see \cite[Proposition 2.4]{maione2022variational}). In the sequel, we denote by $\xi_1$ its corresponding eigenfunction. Then $\xi_1>0$ and satisfies
	\begin{equation}\label{first-eigenvalue}
	\int_{\Omega}\nabla\xi_1\cdot\nabla\phi\ dx+\int_{\R^N}\int_{\R^N}\frac{(\xi_1(x)-\xi_1(y))(\phi(x)-\phi(y))}{|x-y|^{N+2s}}\ dxdy=\lambda_1\int_{\Omega}\xi_1\phi\ dx,\quad\forall\phi\in V.
	\end{equation}
	Consider now $\lambda>0$ such that $u_{\lambda}>0$ is a weak solution of $\eqref{e2}_{\lambda}$ i.e.,
	\begin{equation}\label{j}
	\int_{\Omega}\nabla u_{\lambda}\cdot\nabla\phi\ dx+\int_{\R^N}\int_{\R^N}\frac{(u_{\lambda}(x)-u_{\lambda}(y))(\phi(x)-\phi(y))}{|x-y|^{N+2s}}\ dxdy=\int_{\Omega}u_{\lambda}^{p-1}\phi\ dx+\lambda\int_{\Omega}u_{\lambda}^{q-1}\phi\ dx,
	\end{equation}
	for all $\phi\in V$. Taking $\phi=u_{\lambda}$ and $\phi=\xi_1$ respectively in \eqref{first-eigenvalue} and \eqref{j}, and subtraction \eqref{first-eigenvalue} to \eqref{j}, one obtains that
	\begin{equation}\label{m2}
	\int_{\Omega}u_{\lambda}^{p-1}\xi_1\ dx+\lambda\int_{\Omega}u_{\lambda}^{q-1}\xi_1\ dx = \lambda_1\int_{\Omega}\xi_1u_{\lambda}\ dx.
	\end{equation}
	We now take $\tau=u_{\lambda}$ in \eqref{m1}, multiply \eqref{m1} by $\xi_1$ and integrate over $\Omega$ to get, together with \eqref{m2}, that
	\begin{equation*}
	(\lambda-\tilde{\lambda})\int_{\Omega}u_{\lambda}^{q-1}\xi_1\ dx<0.
	\end{equation*}
	This implies that $\lambda<\tilde{\lambda}$. Thus, $\Lambda\leq\tilde{\lambda}<+\infty$. We now show that there exists a positive solution for every $\lambda\in(0,\Lambda)$. For that, we let $u_{\mu}$ be a solution of $\eqref{e2}_{\mu}$ with $\lambda<\mu<\Lambda$. Then, in particular, $u_{\mu}$ is a supersolution of $\eqref{e2}_{\lambda}$. Consider $w$ as the unique solution (thanks to Lemma \ref{b-k-o-result}) of 
	\begin{equation*}
	-\Delta u+(-\Delta)^su=\lambda u^{q-1},~~~u>0~~\text{in}~~\Omega,\quad u=0~~\text{in}~~\R^N\setminus\Omega.
	\end{equation*}
	From Lemma \ref{b-k-o-result}, we have $w\leq u_{\mu}$. Moreover, since $w$ is a subsolution of $\eqref{e2}_{\lambda}$, then by Lemma \ref{caparison-principle}, there exists $u_{\lambda}\in V$, solution of $\eqref{e2}_{\lambda}$ with $w\leq u_{\lambda}\leq u_{\mu}$. This concludes the proof.
\end{proof}

\section{A non-homogeneous supercritical problem}\label{section:comments}
We would like to mention that a similar existence result to Theorem \ref{first-main-result} can be also established when studying the semilinear problem
\begin{equation}\label{e3}
\left\{\begin{aligned}
-\Delta u+(-\Delta)^su&=|u|^{p-2}u + f~~~~\text{in}~~~\Omega\\
u&=0\quad\quad\quad\quad\quad~~\text{in}~~\R^N\setminus\Omega,
\end{aligned}
\right.
\end{equation}
where $2<p$ and $f\in L^q(\Omega)$ for some $q>N$. In this case, the existence of solutions is assured whenever 
\begin{equation}\label{condition-for-existence}
\|f\|_{L^q(\Omega)}< l.
\end{equation}
for some $l>0$ sufficiently small. Our main result reads as follows.

\begin{thm}
	Let $\Omega\subset\R^N$ be a bounded domain with $C^{1,1}$ boundadry and let $2<p$ and let $f\in L^q(\Omega)$ for some $q>N$. Then there exists $l>0$ such that for $\|f\|_{L^q(\Omega)}<l$, problem \eqref{e3} has a solution in $\cX^1_0\cap W^{2,q}(\Omega)$.
\end{thm}
Notice that the condition imposed on the $L^q$-norm of $f$ is very crucial to establish the pointwise invariance condition $(ii)$ of Theorem \ref{key-theorem-in-the-variational-methods}. 
The corresponding convex set for the variational approach is defined as
\begin{equation}\label{convex-set-2}
K(r):=\{u\in V: \|u\|_{W^{2,q}(\Omega)}\leq r\} 
\end{equation}
where $V$ is the Banach space defined as $V=\cX^1_0\cap W^{2,q}(\Omega)$. The choice of this convex set is dictated by the $W^{2,q}$-regularity theory (for some $q>N$) for the operator $L=-\Delta+(-\Delta)^s$ established recently in \cite[Theorem 1.4 + Lemma 4.4]{su2022regularity}. It states the following.
\begin{thm}(\cite[Theorem 1.4 + Lemma 4.4]{su2022regularity}) 
	Let $\Omega\subset\R^N$ be a $C^{1,1}$ domain and $s\in (0,1),~N>2s$. Then if $f\in L^q(\Omega)$ with $q>N$, then the problem
	\begin{equation}
	Lu=f,\quad\text{in}~~\Omega
	\end{equation}
	has a unique solution $u\in W^{2,q}(\Omega)\cap W^{1,q}_0(\Omega)$. Furthermore,
	\begin{equation}
	\|u\|_{W^{2,q}(\Omega)}\leq C(\|u\|_{L^q(\Omega)}+\|f\|_{L^q(\Omega)})
	\end{equation}
	where $C=C(\Omega,N,s,q)>0$.
\end{thm}

\section{Hamiltonian systems}\label{section:hamiltonian-system}
We consider the following Hamiltonian system
\begin{equation}\label{hamiltonian system}
    \left\{ 
    \begin{aligned}
        Lu&=|v|^{p-2}v~~~~~~~~~~~~~~~~~~~\text{in}~~~\Omega\\
        Lv&=|u|^{d-2}u+\mu |u|^{q-2}u~~~~\text{in}~~~\Omega\\
        u&=v=0~~~~~~~~~~~~~~~~~~~~~~\text{in}~~~\R^N\setminus\Omega
    \end{aligned}
    \right.
\end{equation}
where $\Omega\subset\R^N$ is a bounded $C^{1,1}$ domain and $1<q<2<p, d$. We recall that $L=-\Delta+(-\Delta)^s$ with $s\in(0,1)$. Throughout this section, we consider $s\in (0,\frac{1}{2}]$. Denote by $p'$ the conjugate exponent of $p$. Then $1<p'<2$. We also assume that $N>2p'>2s$. Since $\Omega$ is a bounded $C^{1,1}$ domain in $\R^N$, then it possesses the cone property (see e.g. \cite[Theorem 1.2.2.2]{grisvard2011elliptic}). Therefore, the embedding
\begin{equation}\label{compact-embedding}
    W^{2,p'}(\Omega)\hookrightarrow L^{p'}(\Omega) 
\end{equation}
is compact since $1<p'<p'^*_2:=\frac{Np'}{N-2p'}$.

To prove Theorem \ref{main-result-hamiltonian-system} one may be tempted to find critical points of the corresponding energy function
\begin{equation}
    F(u,v)=\int_{\Omega}L(u) v\ dx-\frac{1}{p}\int_{\Omega}|v|^p\ dx-\frac{1}{d}\int_{\Omega}|u|^{d}\ dx-\frac{\mu}{q}\int_{\Omega}|u|^q\ dx.
\end{equation}
However, the first term on the right-hand side of the above equality is difficult to manipulate. Our strategy is then to find an equivalent problem to \eqref{hamiltonian system}. To this end, we proceed as follows.

Let $(u,v)$ be a weak solution of \eqref{hamiltonian system}. Then from the first equation in \eqref{hamiltonian system}, we have
\begin{equation*}
    v=|Lu|^{\frac{1}{p-1}-1}Lu=|Lu|^{p'-2}Lu.
\end{equation*}
In the latter, we have used that $\frac{1}{p-1}=p'-1$. Now, substituting $v$ by its value in the second equation of \eqref{hamiltonian system}, we obtain the following equivalent problem
\begin{equation}\label{Dirichlet-problem-from-hamiltonian-system}
    \left\{ 
    \begin{aligned}
        L\Big(|Lu|^{p'-2}Lu\Big)&=|u|^{d-2}u+\mu |u|^{q-2}u~~~~\text{in}~~~\Omega\\
        u&=Lu=0~~~~~~~~~~~~~~~~~~~\text{in}~~~\R^N\setminus\Omega.
    \end{aligned}
    \right.
\end{equation}
The corresponding Euler-Lagrange functional energy of \eqref{Dirichlet-problem-from-hamiltonian-system} is the following
\begin{align*}
    I(u)&=\frac{1}{p'}\int_{\Omega}|Lu|^{p'}\ dx-\frac{1}{d}\int_{\Omega}|u|^{d}\ dx-\frac{\mu}{q}\int_{\Omega}|u|^{q}\ dx\\
    &=\Psi(u)-\Phi(u)
\end{align*}
where
\begin{equation*}
    \Psi(u)=\frac{1}{p'}\int_{\Omega}|Lu|^{p'}\ dx\quad\quad\text{and}\quad\quad\Phi(u)=\frac{1}{d}\int_{\Omega}|u|^{d}\ dx+\frac{\mu}{q}\int_{\Omega}|u|^{q}\ dx.
\end{equation*}
In this section, we consider the reflexive Banach space $V=\cX^1_0\cap W^{2,p'}(\Omega)\cap L^{d}(\Omega)$ with the norm
\begin{equation*}
    \|u\|_{V}=\|u\|_{\cX^1_0}+\|u\|_{W^{2,p'}(\Omega)}+\|u\|_{L^{d}(\Omega)}.
\end{equation*}
We also consider the following convex space
\begin{equation*}
    K=K(r)=\{u\in V: \|u\|_{L^{\infty}(\Omega)}\leq r\}.
\end{equation*}
As in the previous sections, $I_K=\Psi_K-\Phi$ denotes the restriction of $I$ on $K$.

\begin{prop}\label{prop-1}
    Let $\ov{u}\in V$ be a critical point of $I_K$ in the sense of Definition \ref{def3}. If there exists $\Tilde{u}\in \cX^1_0\cap W^{2,p'}(\Omega)$ and $\tilde{v}\in K$ such that
    \begin{equation}\label{c-1}
        \left\{
        \begin{aligned}
            L\tilde{u}&=|\Tilde{v}|^{p-2}\Tilde{v}~~~~~~~~~~~~~~~~~~\text{in}~~\Omega\\
            L\tilde{v}&=|\ov{u}|^{d-2}\ov{u}+\mu |\ov{u}|^{q-2}\ov{u}~~~~\text{in}~~\Omega
        \end{aligned}
        \right.
    \end{equation}
    then $(\Tilde{u},\Tilde{v})$ is a solution of
    \begin{equation}\label{c-2}
        \left\{
        \begin{aligned}
            Lu&=|v|^{p-2}v~~~~~~~~~~~~~~~~~~\text{in}~~\Omega\\
            Lv&=|u|^{d-2}u+\mu |u|^{q-2}u~~~~\text{in}~~\Omega
        \end{aligned}
        \right.
    \end{equation}
\end{prop}
Before proving this proposition, we first recall the following well-known result from convex analysis.
\begin{lemma}\label{convex-analysis-result}
    Let $V$ be a reflexive Banach space and let $f:V\to\R$ be a convex and differentiable functional. If
    \begin{equation*}
        f(u)-f(v)\geq\langle Df(u), u-v\rangle,
    \end{equation*}
    then $Df(u)=Df(v)$, where $\langle\cdot,\cdot\rangle$ is the the duality pairing between $V$ and $V^*$. In particular, if $f$ is strictly convex, then $u=v$.
\end{lemma}
We now give the proof of Proposition \ref{prop-1}.
\begin{proof}[Proof of Proposition \ref{prop-1}]
    From the first equation in \eqref{c-1}, we have
    \begin{equation}\label{d-1}
        \tilde{v}=|L\tilde{u}|^{p'-2}L\tilde{u}.
    \end{equation}
    Now, we introduce the functional $J:\cX^1_0\cap W^{2,p'}(\Omega)\to \R$ defined as follows
    \begin{equation*}
        J(w)=\frac{1}{p'}\int_{\Omega}|Lw|^{p'}\ dx-\frac{1}{d}\int_{\Omega}|\ov u|^{d}-\frac{\mu}{q}\int_{\Omega}|\ov u|^{q}\ dx.
    \end{equation*}
    Then $\tilde{u}$ is a critical point of $J$. Indeed, for every $\phi\in \cX^1_0\cap W^{2,p'}(\Omega)$, 
    \begin{align*}
        \langle J'(\tilde{u}),\phi\rangle&=\int_{\Omega}|L\tilde{u}|^{p'-2}L\tilde{u}L\phi\ dx-\int_{\Omega}(|\ov u|^{d-2}\ov u+\mu |\ov u|^{q-2}\ov u)\phi\ dx\\
        &=\int_{\Omega}\tilde{v}L\phi\ dx-\int_{\Omega}(|\ov u|^{d-2}\ov u+\mu |\ov u|^{q-2}\ov u)\phi\ dx\\
        &=\int_{\Omega}(L\tilde{v})\phi\ dx-\int_{\Omega}(|\ov u|^{d-2}\ov u+\mu |\ov u|^{q-2}\ov u)\phi\ dx\\
        &=\int_{\Omega}(|\ov u|^{d-2}\ov u+\mu |\ov u|^{q-2}\ov u)\phi\ dx-\int_{\Omega}(|\ov u|^{d-2}\ov u+\mu |\ov u|^{q-2}\ov u)\phi\ dx\\
        &=0.
    \end{align*}
    Thus, $\tilde{u}$ is a critical point of $J$. Notice that in the second equality, we have used \eqref{d-1}, while in the fourth equality, \eqref{c-1} has been used. Now, taking in particular $\phi=\tilde{u}-\ov u$, we infer that
    \begin{align*}
        0=\langle J'(\tilde{u}),\tilde{u}-\ov u\rangle=\int_{\Omega}|L\tilde{u}|^{p'-2}L\tilde{u}L(\tilde{u}-\ov u)\ dx-\int_{\Omega}(|\ov u|^{d-2}\ov u+\mu |\ov u|^{q-2}\ov u)(\tilde{u}-\ov u)\ dx
    \end{align*}
    that is,
    \begin{align}\label{d-2}
        \nonumber\int_{\Omega}|L\tilde{u}|^{p'-2}L\tilde{u}L(\tilde{u}-\ov u)\ dx&=\int_{\Omega}(|\ov u|^{d-2}\ov u+\mu |\ov u|^{q-2}\ov u)(\tilde{u}-\ov u)\ dx\\
        &=\int_{\Omega}D\Phi(\ov u)(\tilde{u}-\ov u)\ dx.
    \end{align}
     On the other hand, since $\ov{u}\in V$ is a critical point of $I_K$ in the sense of Definition \ref{def3}, then
    \begin{equation*}
        \frac{1}{p'}\int_{\Omega}|L\phi|^{p'}\ dx-\frac{1}{p'}\int_{\Omega}|L\ov{u}|^{p'}\ dx\geq\int_{\Omega}D\Phi(\ov{u})(\phi-\ov{u})\ dx~~~~\forall\phi\in V.
    \end{equation*}
In particular, $\phi=\tilde{u}$ gives
\begin{equation}\label{d-3}
        \frac{1}{p'}\int_{\Omega}|L\tilde{u}|^{p'}\ dx-\frac{1}{p'}\int_{\Omega}|L\ov{u}|^{p'}\ dx\geq\int_{\Omega}D\Phi(\ov{u})(\tilde{u}-\ov{u})\ dx.
    \end{equation}
From \eqref{d-2} and \eqref{d-3}, we get
\begin{equation*}
    \frac{1}{p'}\int_{\Omega}|L\tilde{u}|^{p'}\ dx-\frac{1}{p'}\int_{\Omega}|L\ov{u}|^{p'}\ dx\geq \int_{\Omega}|L\tilde{u}|^{p'-2}L\tilde{u}L(\tilde{u}-\ov u)\ dx.
\end{equation*}
Since the function $t\mapsto\frac{1}{p'}|t|^{p'}$ is strictly convex, then by Lemma \ref{convex-analysis-result}, we have that $\tilde{u}=\ov u$. The proof is therefore finished by considering $\tilde{u}=\ov u$ in \eqref{c-1}.
\end{proof}

\begin{lemma}\label{lm2-2}
	Suppose that
	\begin{equation*}
	K=\{u\in K(r): u(x)\geq0~~\text{for almost all}~~x\in\Omega\}.
	\end{equation*}
	Then there exists $\ov{u}\in K$ such that $I_{K}(\ov{u})=\inf_{u\in V}I_{K}(u)$.
\end{lemma}

\begin{proof}
    Let $\beta=\inf_{u\in V}I_K(u)$. Recalling the definition of $\Psi_K$, we have $I_K(u)=+\infty$ for $u\notin K$. Now, it is not difficult to see that $\beta>-\infty$. In fact, for $u\in K$,
    \begin{align*}
        \Phi(u)&=\frac{1}{d}\int_{\Omega}|u|^d\ dx+\frac{\mu}{q}\int_{\Omega}|u|^q\ dx\\
        &\leq c_1\|u\|^d_{L^{\infty}(\Omega)}+c_2\|u\|^q_{L^{\infty}(\Omega)}\\
        &\leq c_1r^d+c_2r^q.
    \end{align*}
Hence,
\begin{equation*}
    I_K(u)=\Psi_K(u)-\Phi(u)\geq-(c_1r^d+c_2r^q)>-\infty
\end{equation*}
and therefore $\beta>-\infty$. Note that in the first inequality, we have that $\Psi_K$ is non-negative.

We wish now to show that the infimum $\beta$ is achieved. Let $\{u_n\}$ be a minimizing sequence of $I_K$ in $K$, namely $u_n\in K$ is such that $I_K(u_n)\to \beta$. Then $\{I_K(u_n)\}$ is bounded. Moreover, since also $\{\Phi(u_n)\}$ is bounded (thanks to the fact that $u_n\in K$), we then deduce that $\{\Psi_K(u_n)\}$ is bounded as well. This means that there is $C_1>0$ such that
\begin{equation*}
    \Psi_K(u_n)\leq C_1~~~~~\forall n,
\end{equation*}
that is,
\begin{equation}\label{o7}
    \|Lu_n\|_{L^{p'}(\Omega)}\leq C_2~~~~\forall n.
\end{equation}
%Moreover, $g(u_n)=|u_n|^{d-2}u_n+\mu|u_n|^{q-2}u_n$ is uniformly bounded (because $u_n\in K$). Therefore, by Lemma \ref{l-infty-estimate}, $|Lu_n|^{p'-2}Lu_n\in L^{\infty}(\Omega)$ and 
%\begin{equation}\label{o8}
%    \||Lu_n|^{p'-2}Lu_n\|_{L^{\infty}(\Omega)}\leq C\|g(u_n)\|_{L^{\infty}(\Omega)}\leq C'~~~~~\forall n 
%\end{equation}
%for some $C, C'>0$ independent of $n$.
We now set
\begin{equation}\label{o9}
    w_n=|Lu_n|^{p'-2}Lu_n.
\end{equation}
It is easily seen that $u_n$ satisfies
\begin{equation}
    \left\{ 
    \begin{aligned}
        Lu_n&=|w_n|^{p-2}w_n~~~~~\text{in}~~~\Omega\\
        u_n&=0~~~~~~~~~~~~~~~~~\text{in}~~~\R^N\setminus\Omega.
    \end{aligned}
    \right.
\end{equation}
Notice that by \eqref{o7}, $|w_n|^{p-2}w_n\in L^{p'}(\Omega)$ with
\begin{equation*}
    \||w_n|^{p-2}w_n\|_{L^{p'}(\Omega)}=\|Lu_n\|_{L^{p'}(\Omega)}\leq C_2~~~~~~\forall n.
\end{equation*}
Hence, by \cite[Theorem 1.4]{su2022regularity} there exists a positive constant $C=C(\Omega,N,s,p')$ such that
\begin{align*}
    \|u_n\|_{W^{2,p'}(\Omega)}&\leq C(\|u_n\|_{L^{p'}(\Omega)}+\||w_n|^{p-2}w_n\|_{L^{p'}(\Omega)})\\
    &\leq C(c_1\|u_n\|_{L^{\infty}(\Omega)}+\|Lu_n\|_{L^{p'}(\Omega)})\\
    &\leq C(c_1r+C_2)~~~~~~~~~~\forall n.
\end{align*}
This implies that $u_n$ is bounded in $W^{2,p'}(\Omega)$. We also deduce that $u_n$ is bounded in $\cX^1_0$. Indeed,
\begin{align}\label{p_3}
  \nonumber  \|\nabla u_n\|^2_{L^2(\Omega)}&=\int_{\Omega}|\nabla u_n|^2\ dx=\int_{\Omega}(-\Delta u_n)u_n\ dx\\
  \nonumber  &\leq \int_{\Omega}|\Delta u_n||u_n|\ dx
    \leq\|u_n\|_{L^{p}(\Omega)}\|\Delta u_n\|_{L^{p'}(\Omega)}\\
    &\leq |\Omega|^{\frac{1}{p}}r\|\Delta u_n\|_{L^{p'}(\Omega)}\leq |\Omega|^{\frac{1}{p}}r\|u_n\|_{W^{2,p'}(\Omega)} \leq C~~~~~\forall n. 
\end{align}
Now, thanks to \eqref{p_1}, it follows from \eqref{p_3} that $u_n$ is bounded in $\cX^1_0$. Thus, after passing to a subsequence, there is $\ov{u}\in \cX^1_0\cap W^{2,p'}(\Omega)$ such that
\begin{align*}
    &u_n\rightharpoonup \ov{u}~~~\text{weakly in}~~\cX^1_0\cap W^{2,p'}(\Omega)\\
    &u_n\to \ov{u}~~~\text{strongly in}~~L^{p'}(\Omega).
\end{align*}
Notice that in the latter, we have used the compact embedding \eqref{compact-embedding}. In particular, $u_n\to u$ a.e. in $\Omega$. Thus, $\|u\|_{L^{\infty}(\Omega)}\leq r$. On the other hand, $u_n\to u$ strongly in $L^{d}(\Omega)$. Indeed,
\begin{align*}
    \|u_n-\ov{u}\|^d_{L^d(\Omega)}&=\int_{\Omega}|u_n-\ov{u}|^d\ dx\\
    &=\int_{\Omega}|u_n-\ov{u}|^{d-p'}|u_n-\ov{u}|^{p'}\ dx\\
    &\leq (\|u_n\|^{d-p'}_{L^{\infty}(\Omega)}+\|\ov{u}\|^{d-p'}_{L^{\infty}(\Omega)})\|u_n-\ov{u}\|^{p'}_{L^{p'}(\Omega)}\\
    &\leq (2r)^{d-p'}\|u_n-\ov{u}\|^{p'}_{L^{p'}(\Omega)}\to 0~~~~\text{as}~~~ n\to \infty.
\end{align*}
From the above results, we have that $\ov{u}\in V$. Moreover, since $\|\ov{u}\|_{L^{\infty}(\Omega)}\leq r$, we deduce that $\ov u\in K$. Thus the proof is completed if we prove that $I_K(\ov{u})\leq \beta$. To this end, let us first notice that $u_n\to \ov{u}$ strongly in $L^q$ (since $q<d$) then by dominated convergence theorem,
\begin{equation*}
    \Phi(u_n)\to\Phi(\ov{u})~~~~\text{as}~~~n\to\infty.
\end{equation*}
We now wish to show that $\Psi_K$ is lower semi-continuous i.e.,
\begin{equation}\label{o6-1}
    \Psi_K(\ov{u})\leq \liminf_{n\to\infty}\Psi_K(u_n).
\end{equation}
Notice first that 
\begin{equation*}
    Lu_n \rightharpoonup L\ov{u}~~~~\text{weakly in}~~L^{p'}(\Omega)
\end{equation*}
that is,
\begin{equation}\label{o10}
    \lim_{n\to\infty}\int_{\Omega}(Lu_n-L\ov{u})\phi\ dx=0~~~~~~~\forall\phi\in L^p(\Omega).
\end{equation}
Now recalling that
\begin{equation*}
    \Psi_K(u_n)-\Psi_K(\ov{u})=\frac{1}{p'}\int_{\Omega}|Lu_n|^{p'}\ dx-\frac{1}{p'}\int_{\Omega}|L\ov{u}|^{p'}\ dx,
\end{equation*}
we use the convexity of the function $t\mapsto \frac{1}{p'}|t|^{p'}$ to see that
\begin{equation*}
    \Psi_K(u_n)-\Psi_K(\ov{u})\geq \int_{\Omega}|L\ov u|^{p'-2}L\ov u(Lu_n-L\ov{u})\ dx.
\end{equation*}
Since $|L\ov u|^{p'-2}L\ov u\in L^p(\Omega)$, we then let $n\to\infty$ in the above inequality to get \eqref{o6-1}, thanks to \eqref{o10}. Thus,
\begin{equation*}
    I_K(\ov{u})=\Psi_K(\ov{u})-\Phi(\ov{u})\leq\liminf_{n\to\infty}(\Psi_K(u_n)-\Phi(u_n))=\liminf_{n\to\infty}I_{K}(u_n)=\beta.
\end{equation*}
This completes the proof.
\end{proof}

\begin{lemma}\label{lm5-1}
	If $1<q<2<p$, there exists $\mu_*>0$ such that for every $\mu\in(0,\mu_*)$, there exist $r_1, r_2\in\R$ with $r_1<r_2$ such that for every $r\in [r_1, r_2]$ and every $\ov{u}\in K(r)$, the problem
	\begin{equation}
    \left\{
    \begin{aligned}
     L\tilde{u}&=|\tilde{v}|^{p-2}\tilde{v}~~~~~~~~~~~~~~~~~~~\text{in}~~~\Omega \\
     L\tilde{v}&=|\ov{u}|^{d-2}\ov{u}+\mu|\ov{u}|^{q-2}\ov{u}~~~~~\text{in}~~~\Omega
    \end{aligned}
    \right.
	\end{equation}
	has a weak solution $(\tilde{u}, \tilde{v})$ with $\tilde{u}\in \cX^1_0\cap W^{2,p'}(\Omega)\cap L^{\infty}(\Omega)$ and $\tilde{v}\in K(r)$. 
\end{lemma}

\begin{proof}
    Consider $\ov{u}\in K(r)$. Then $\|\ov{u}\|_{L^{\infty}(\Omega)}\leq r$. In particular, $\ov{u}\in L^{\infty}(\Omega)$. Thus, we deduce from Lemma \ref{l-infty-estimate} the existence of $\tilde{v}\in\cX^1_0$ solving in the weak sense the equation
    \begin{equation*}
        L\tilde{v}=|\ov{u}|^{p-2}\ov{u}+\mu|\ov{u}|^{q-2}\ov{u}~~~~~\text{in}~~~\Omega.
    \end{equation*}
    Furthermore, $\tilde{v}\in L^{\infty}(\Omega)$ with
    \begin{align*}
        \|\tilde{v}\|_{L^{\infty}(\Omega)}&\leq C\||\ov{u}|^{d-2}\ov{u}+\mu |\ov{u}|^{q-2}\ov{u}\|_{L^{\infty}(\Omega)}\\
        &\leq C(\|\ov{u}\|^{d-1}_{L^{\infty}(\Omega)}+\mu\|\ov{u}\|^{q-1}_{L^{\infty}(\Omega)})\\
        &\leq C(r^{d-1}+\mu r^{q-1}).
    \end{align*}
It then follows from the inequality above that $\tilde{v}\in K(r)$ if and only if $C(r^{d-1}+\mu r^{q-1})\leq r$. A simple analysis of the function $r\mapsto \alpha(r)=C(r^{d-2}+\mu r^{q-2})-1$ shows that there exists $\mu_*>0$ such that for all $\mu\in(0,\mu_*)$, there exist $r_1<r_2$ such that for all $r\in [r_1,r_2]$, $\alpha(r)\leq0$. On the other hand, by \cite[Theorem 1.4]{su2022regularity}, we have $\tilde{v}\in W^{2,p'}(\Omega)$. Hence, $\tilde{v}\in K(r)$.

Now, using again Lemma \ref{l-infty-estimate} (thanks to the fact that $\tilde{v}\in L^{\infty}(\Omega)$), we find $\tilde{u}\in\cX^1_0$ satisfying in the weak sense the equation 
\begin{equation*}
    L\tilde{u}=|\tilde{v}|^{p-2}\tilde{v}~~~~~~\text{in}~~~\Omega
\end{equation*}
and $\tilde{u}\in L^{\infty}(\Omega)$. Hence, by \cite[Theorem 1.4]{su2022regularity} it follows that $\tilde{u}\in W^{2,p'}(\Omega)$. Thus, $\tilde{u}\in \cX^1_0\cap W^{2,p'}(\Omega)$. This completes the proof. 
\end{proof}
We now prove the main result of this section.

\begin{proof}[Proof of Theorem \ref{main-result-hamiltonian-system}]
 Let $r_1, r_2$ and $\mu_*$ be as in Lemma \ref{lm5-1}. Set
 \begin{equation*}
     K:=\{u\in K(r): u(x)\geq0~~\text{a.e. in}~~\Omega\}
 \end{equation*}
 for some $r\in [r_1, r_2]$. Then by Lemma \ref{lm2-2} there exists $\ov{u}\in K$ such that $I_K(\ov{u})=\inf_{u\in V}I_K(u)$. We claim that $\ov u$ is a nontrivial critical point of $I_K$ in the sense of Definition \ref{def3}. Indeed, since $\tilde{u}$ is a minimizer of $I_K$ in $V$, then
	\begin{align*}
	I_K(\tilde{u})\leq I_K((1-t)\tilde{u}+tv)\quad\quad \forall v\in V
	\end{align*} 
	for all sufficiently small $t>0$.
	Recalling that $I_K=\Psi_K-\Phi$, then from the inequality above and using that $\Psi_K$ is convex, we have
	\begin{align*}
	0\leq I_K((1-t)\tilde{u}+tv)-I_K(\tilde{u})&=\Phi(\tilde{u})-\Phi((1-t)\tilde{u}+tv)+\Psi_K((1-t)\tilde{u}+tv)-\Psi_K(\tilde{u})\\
	&\leq \Phi(\tilde{u})-\Phi(\tilde{u}+t(v-\tilde{u}))+t(\Psi_K(v)-\Psi_K(\tilde{u})).
	\end{align*}
	The identity \eqref{critical-point} then follows by dividing the inequality above by $t$ and letting $t\rightarrow0$.

 Regarding the non-triviality of $\ov u$, we notice first that for $u_0\in K$, we have $tu_0\in K$ for every $t\in [0,1]$. Then
 \begin{align*}
     I_K(tu_0)&=\frac{t^{p'}}{p'}\int_{\Omega}|Lu_0|^{p'}\ dx-\frac{t^d}{d}\int_{\Omega}|u_0|^d-\frac{\mu t^{q}}{q}\int_{\Omega}|u_0|^{q}\ dx \\
     &=t^{q}\Big(\frac{t^{p'-q}}{p'}\int_{\Omega}|Lu_0|^{p'}\ dx-\frac{t^{d-q}}{d}\int_{\Omega}|u_0|^d-\frac{\mu}{q}\int_{\Omega}|u_0|^{q}\ dx\Big).
 \end{align*}
 For $t$ sufficiently small and using that $q<p'$, we have $I_K(tu_0)<0$ and thus $I_K(\ov{u})<0$. This implies that $\ov{u}$ is a nontrivial critical point of $I_K$.

 Now, by Lemma \ref{lm5-1}, there is $(\tilde{u},\tilde{v})\in (\cX^1_0\cap W^{2,p'}(\Omega))\times K(r)$ satisfying in the weak sense
 \begin{equation*}
     \left\{
     \begin{aligned}
         L\tilde{u}&=|\tilde{v}|^{p-2}\tilde{v}~~~~~~~~~~~~~~~~~~~~\text{in}~~~\Omega\\
         L\tilde{v}&=|\ov{u}|^{d-2}\ov{u}+\mu|\ov{u}|^{q-2}\ov{u}~~~~~\text{in}~~~\Omega.
     \end{aligned}
     \right.
 \end{equation*}
 Moreover, since $\ov{u}\geq0$ in $\Omega$, then from Proposition \ref{maximum-principle}, $\tilde{v}\geq0$ in $\Omega$ and thus, $\tilde{v}\in K$. Hence, by Proposition \ref{prop-1}, $(\tilde{u},\tilde{v})$ is a solution of \eqref{hamiltonian system}. It then remains to show that $(\tilde{u},\tilde{v})$ is positive. Notice first that from the non-negativity of $\tilde{v}$, it follows that $\tilde{u}\geq0$ in $\Omega$, thanks to Proposition \ref{maximum-principle}. Now, $\tilde{u}>0$ and $\tilde{v}>0$ in $\Omega$ follows from the strong maximum principle. 
\end{proof}

\section*{Data availability statement}
Data sharing not applicable to this article as no datasets were generated or analyzed during the current study.

\section*{Declaration of competing interest}

The authors declare that they have no known competing financial interests or personal relationships that could have
appeared to influence the work reported in this paper.

\section*{Authors contributions} 

All authors contributed equally. \\

\textbf{Acknowledgements:} D.A. and A.M. are  pleased to acknowledge the support of the Natural Sciences and Engineering Research Council of Canada. R.Y.T. is supported by Fields Institute. The authors would like to thank the referee for valuable comments and suggestions.

\bibliographystyle{ieeetr}

\begin{thebibliography}{10}

 

\bibitem{abatangelo2021elliptic} N. Abatangelo and M. Cozzi, \emph{An elliptic boundary value problem with fractional nonlinearity.} SIAM Journal on Mathematical Analysis 53.3 (2021): 3577-3601.

\bibitem{allegretto1998picone} W. Allegretto and H. Y. Xi, \emph{A Picone's identity for the $p$-Laplacian and applications.} Nonlinear Analysis: Theory, Methods \& Applications 32.7 (1998): 819-830.

\bibitem{ABC} A. Ambrosetti, H. Brezis and G. Cerami, \emph{Combined effects of concave and convex nonlinearities in some elliptic problems}, J. Funct. Anal. 122 (1994) 519-543.

\bibitem{anthal2022choquard} G. C. Anthal, J. Giacomoni, and K. Sreenadh, \emph{Choquard equation involving mixed local and nonlocal operators.} arXiv preprint arXiv:2212.07760 (2022).

\bibitem{barles2012lipschitz} G. Barles, E. Chasseigne, A. Ciomaga, and C. Imbert, \emph{Lipschitz regularity of solutions for mixed integro-differential equations.} Journal of differential equations 252.11 (2012): 6012-6060.

\bibitem{biagi2022brezis} S. Biagi, S. Dipierro, E. Valdinoci, and E. Vecchi, \emph{A Brezis-Nirenberg type result for mixed local and nonlocal operators.} arXiv preprint arXiv:2209.07502 (2022).

\bibitem{biagi2021faber} S. Biagi, S. Dipierro, E. Valdinoci, and E. Vecchi, \emph{ A Faber-Krahn inequality for mixed local and nonlocal operators.} Journal d'Analyse Math\'{e}matique (2023): 1-43.

\bibitem{biagi2022hong} S. Biagi, S. Dipierro, E. Valdinoci, and E. Vecchi, \emph{A Hong-Krahn-Szeg\"{o} inequality for mixed local and nonlocal operators.} Mathematics In Engineering 5.1 (2022).
		
\bibitem{biagi2022mixed} S. Biagi, S. Dipierro, E. Valdinoci, and E. Vecchi, \emph{Mixed local and nonlocal elliptic operators: regularity and maximum principles.} Communications in Partial Differential Equations 47.3 (2022): 585-629.

\bibitem{biagi2021semilinear} S. Biagi, S. Dipierro, E. Valdinoci, and E. Vecchi, \emph{Semilinear elliptic equations involving mixed local and nonlocal operators.} Proceedings of the Royal Society of Edinburgh Section A: Mathematics 151.5 (2021): 1611-1641.

\bibitem{biagi2021brezis} S. Biagi, D. Mugnai, and E. Vecchi, \emph{A Brezis-Oswald approach for mixed local and nonlocal operators.} Commun. Contemp. Math., Article 2250057 (2022): 1-28.

\bibitem{biswas2021mixed} A. Biswas, and M. Modasiya, \emph{Mixed local-nonlocal operators: maximum principles, eigenvalue problems and their applications.} arXiv preprint arXiv:2110.06746 (2021).

\bibitem{biswas2022boundary} A. Biswas, M. Modasiya, and A. Sen, \emph{Boundary regularity of mixed local-nonlocal operators and its application.} Annali di Matematica Pura ed Applicata (1923-) (2022): 1-32.

\bibitem{brasco2014convexity} L. Brasco and G. Franzina, \emph{Convexity properties of Dirichlet integrals and Picone-type inequalities.} Kodai Mathematical Journal 37.3 (2014): 769-799.

\bibitem{byun2023regularity} S.-S. Byun, H.-S. Lee, and K. Song, \emph{Regularity results for mixed local and nonlocal double phase functionals.} arXiv preprint arXiv:2301.06234 (2023).

\bibitem{byun2023mixed} S.-S. Byun and K. Song, \emph{Mixed local and nonlocal equations with measure data.} Calculus of Variations and Partial Differential Equations 62.1 (2023): 14.

\bibitem{de2022gradient} C. De Filippis and G. Mingione, \emph{Gradient regularity in mixed local and nonlocal problems.} Mathematische Annalen (2022): 1-68.

\bibitem{di2012hitchhikers} E. Di Nezza, G. Palatucci, and E. Valdinoci, \emph{Hitchiker's guide to the fractional Sobolev spaces.} Bulletin des sciences math\'{e}matiques 136.5 (2012): 521-573.

\bibitem{dipierro2022linear} S. Dipierro, E. Proietti Lippi, and E. Valdinoci, \emph{Linear theory for a mixed operator with Neumann conditions} Asymptotic Analysis 128.4 (2022): 571-594.

\bibitem{dipierro2021description} S. Dipierro and E. Valdinoci, \emph{Description of an ecological niche for a mixed local/nonlocal dispersal: an evolution equation and a new Neumann condition arising from the superposition of Brownian and L\'{e}vy processes.} Physica A: Statistical Mechanics and its Applications 575 (2021): 126052.

\bibitem{ekeland1976convex} I. Ekeland and R. Temam, \emph{Convex Analysis and Variational Problems}, American Elsevier Publishing Co., Inc., New York, 1976. 

\bibitem{fang2022regularity} Y. Fang, B. Shang, and C. Zhang, \emph{Regularity theory for mixed local and nonlocal parabolic p-Laplace equations.} The Journal of Geometric Analysis 32.1 (2022): 22.

\bibitem{garain2022regularity} P. Garain and J. Kinnunen, \emph{On the regularity theory for mixed local and nonlocal quasilinear elliptic equations.} Transactions of the American Mathematical Society 375.08 (2022): 5393-5423.

\bibitem{garain2023higher} P. Garain and E. Lindgren, \emph{Higher H\"{o}lder regularity for mixed local and nonlocal degenerate elliptic equations.} Calculus of Variations and Partial Differential Equations 62.2 (2023): 67.

\bibitem{garain2022mixed} P. Garain and A. Ukhlov, \emph{Mixed local and nonlocal Sobolev inequalities with extremal and associated quasilinear singular elliptic problems.} Nonlinear Analysis 223 (2022): 113022.

\bibitem{grisvard2011elliptic} P. Grisvard, \emph{Elliptic problems in nonsmooth domains.} Society for Industrial and Applied Mathematics, 2011.

\bibitem{kouhestani2019multiplicity} N. Kouhestani, H. Mahyar, and A. Moameni, \emph{Multiplicity results for a non-local problem with concave and convex nonlinearities.} Nonlinear Analysis 182 (2019): 263-279.

\bibitem{kouhestani2018multiplicity} N. Kouhestani and A. Moameni, \emph{Multiplicity results for elliptic problems with super-critical concave and convex nonlinearities.} Calculus of Variations and Partial Differential Equations 57 (2018): 1-12.

\bibitem{kuratowski1958topologie} K. Kuratowski, \emph{Topologie I}, PWN, Warsaw, 1958.

\bibitem{leonori2015basic} T. Leonori, I. Peral, A. Primo, and F. Soria, \emph{Basic estimates for solutions of a class of nonlocal elliptic and parabolic equations.} Discrete \& Continuous Dynamical Systems 35.12 (2015): 6031.

\bibitem{Li} X. Li, S. Huang, M. Wu,   C. Huang,  {\emph Existence of solutions to elliptic equation with mixed local and nonlocal operators.} AIMS Mathematics, 2022, Volume 7, Issue 7: 13313-13324.

\bibitem{maione2022variational} A. Maione, D. Mugnai, and E. Vecchi, \emph{Variational methods for nonpositive mixed local-nonlocal operators.} Fractional Calculus and Applied Analysis (2023): 1-19.

\bibitem{moameni2018variational} A. Moameni, \emph {Critical point theory on convex subsets with applications in differential equations and analysis.} J. Math. Pures Appl. (9) 141 (2020), 266-315.  

\bibitem{moameni2017variational} A. Moameni, \emph{A variational principle for problems with a hint of convexity.} Comptes Rendus Mathematique 355.12 (2017): 1236-1241.

\bibitem{ambrosetti1983some} P.H. Rabinowitz, 1983. \emph{ Some aspects of critical point theory}. in: MRC Tech, Rep., Madison, Wisconsin.

\bibitem{rabinowitz1974variational}  P.H. Rabinowitz, \emph{Variational methods of nonlinear eigenvalue problems}, in: Proc. Sym. on Eigenvalues of Nonlinear Problems, Edizionicremonese, Rome, 1974, pp. 139--195.

\bibitem{salort2022mixed} A. M. Salort and E. Vecchi, \emph{On the mixed local-nonlocal H\'{e}non equation.} Differential and Integral Equations 35.11/12 (2022): 795-818.

\bibitem{su2022regularity} X. Su, E. Valdinoci, Y. Wei, and J. Zhang, \emph{Regularity results for solutions of mixed local and nonlocal elliptic equations.} Mathematische Zeitschrift 302.3 (2022): 1855-1878.

\bibitem{szulkin1986minimax} A. Szulkin, \emph{Minimax principles for lower semicontinuous functions and applications to nonlinear boundary value problems.} Annales de l'Institut Henri Poincar\'{e} C, Analyse non lin\'{e}aire. Vol. 3. No. 2., 1986.
	
	
	
\end{thebibliography}

\end{document}